\documentclass[reqno]{amsart}

\newcommand{\R}{\mathbb{R}}         
\newcommand{\Z}{\mathbb{Z}}         

\newcommand{\eps}{\epsilon}         

\newcommand{\BigOh}{\mathcal{O}}

\providecommand{\norm}[1]{\lVert#1\rVert}

\newcommand{\df}{\Delta f}

\usepackage{graphicx}

\usepackage{color}                                  
\usepackage[normalem]{ulem}                
\usepackage{comment}				
\usepackage[numbers, square, comma, sort&compress]{natbib}     

\usepackage[normalem]{ulem}			

\usepackage{pgfplots}

\newtheorem{theorem}{Theorem}[section]
\newtheorem{rmk}{Remark}
\newtheorem{lemma}[theorem]{Lemma}
\newtheorem{cor}[theorem]{Corollary}

\usepackage{hyperref}

\author{Mathew F. Causley}
\address{Kettering University,
Department of Mathematics,
1700 University Ave.
Flint, MI 48504, USA}
\email{mcausley@kettering.edu}

\author[D.C. Seal]{David C. Seal}
\address{U.S. Naval Academy,
    Department of Mathematics, 
    572C Holloway Road,
    Annapolis, MD 21402, USA}
\email{seal@usna.edu}

\begin{document}

\title{On the convergence of spectral deferred correction methods}

\date{Last modified: \today}
\maketitle

\begin{abstract}
In this work we analyze the convergence properties of the Spectral Deferred Correction (SDC)
method originally proposed by Dutt et al.\,[{\it BIT}, {40} (2000), pp. 241--266].
The framework for this high-order ordinary differential
equation (ODE) solver
is typically described wherein a low-order approximation (such as forward or backward Euler)
is lifted to higher order accuracy by applying the \emph{same} low-order method to an error equation and then adding in the resulting defect to correct
the solution.
Our focus is not on solving the error equation to increase the order of accuracy, but on rewriting the solver as an iterative Picard integral equation solver.
In doing so, our chief finding is that it is not the low-order solver that picks up the order of accuracy with each correction,
but it is the underlying quadrature rule of the right hand side function that is solely
responsible for picking up additional orders of accuracy.
Our proofs point to a total of three sources of errors that SDC methods carry: the error at the current time point,
the error from the previous iterate, and the numerical integration error that comes from the total number of quadrature nodes
used for integration.
The second of these two sources of errors is what separates SDC methods from Picard integral equation methods; our findings indicate that as long as difference between the current and previous iterate always gets multiplied by at least a constant multiple of the time step size, then high-order accuracy can be found even if the underlying ``solver" is inconsistent the underlying ODE.
From this vantage, we solidify the prospects of extending spectral deferred correction methods to a larger class of solvers to which we present
some examples.
\end{abstract}

\section{Introduction}

The spectral deferred correction (SDC) method defines a large class of ordinary differential equation (ODE) solvers that were originally introduced in 2000 by Dutt, Greengard and
Rokhlin \cite{DuttGreenRokh00}.
These type of methods are typically introduced by defining an \emph{error equation}, and then repeatedly applying the same low-order solver
to the error equation and adding the solution back into the current approximation in order to pick up an order of accuracy.
This idea can be traced back to the work of 
Zadunaisky in 1976 \cite{Zadu76}, where the author sought out high-order solvers in order to reduce numerical 
roundoff errors for astronomical applications.
Before introducing the classical SDC methods defined in \cite{DuttGreenRokh00}, we stop here to point out some of the recent work that has been happening over the past two decades
including \cite{Minion03,LaytonMinion05,ChriOngQiu09,Layton09,ChriOngQiu10,ChriMorOngQiu11,JiaHillEvansFannTaylor13}.
We refer the interested reader to \cite{Morton10} for a nice list of references
for the first of these last two decades.  Here, we provide a sampling of some of the current topics of interest to the community.

Many variations of the original SDC method are being studied as part of an effort to
expedite the convergence of the solver.  The chief goal here is to reduce the total number of iterations required to obtain the 
same high-order accuracy of the original method.  These methods
include the option of using Krylov deferred correction methods \cite{HuangJiaMinion06,HuangJiaMinion07} as well as the
\emph{multi-level} SDC methods  \cite{LaytonMinion04,SpeRuEmMiBoKr15}.
The multi-level approach starts with a lower order interpolant and then 
successively increases the degree of the interpolant with each future sweep of the method.
This has 
the primary advantage of decreasing the overall number of function evaluations that need to be conducted, but introduces additional complications
involving the need to evaluate interpolating polynomials.
In the same vein, higher order embedded integrators have been explored within the so-called integral deferred correction (IDC) framework
 \cite{ChriOngQiu09,ChriOngQiu10}, where a moderate order solver (such as second or fourth-order Runge-Kutta method) is embedded inside a very high
 order SDC solver.  With this framework, each successive correction increases the order by the same amount as that of the base solver.
%
%
%
In addition, parallel in time solvers \cite{ChrMacOng10,Minion10,ChriOng11,EmmettMinion12} are being investigated as a mechanism
to address the needs of modern high performance computing architectures, and
adaptive time stepping options have been more recently investigated in
\cite{ChriMacOngSpi15}.  This work is based upon the nice property that SDC methods naturally embed a lower order solver inside a higher order solver.

In addition to the above mentioned extensions, various semi-implicit formulations have, and are currently being explored.  While the original solver was meant for classical non-linear ODEs, semi-implicit formulations have been conducted as early as 2003 \cite{Minion03} and are still
an ongoing topic of research \cite{LaytonMinion07,ChriMorOngQiu11}.
The effect of the choice of correctors including second order semi-implicit solvers for the error equation has been conducted in \cite{Layton08},
and an investigation into the efficiency of semi-implicit and multi-implicit spectral deferred correction methods for problems with varying temporal scales 
have been conducted in \cite{Layton09}.
Related high-order operator splitting methods have been proposed in \cite{HagZhou06,ChriGuoMorQiu14,ChriLiuXu15}, where the focus is not on
an implicit-explicit splitting, but rather on splitting the right hand side of the ODE into smaller systems that can be more readily inverted with each sweep
of the solver.

Very recent work includes applications of the SDC framework to generate exponential integrators of arbitrary orders
 \cite{Buvoli15},
exploring interesting LU decompositions of the implicit Butcher tableau on non-equispaced grids \cite{Weiser15},
further investigation into high-order operator splitting \cite{DuarteEmmett16}, a comparison of essentially non-oscillatory (ENO) versus piecewise parabolic method (PPM) coupled
with SDC time integrators \cite{KadCol16}, and additional implicit-explicit (IMEX) splittings for fast-wave slow-wave splitting constructed from within the SDC framework \cite{RuSp16}.

It is not our aim to conduct a comprehensive review and comparison of all of these methods, rather it is our goal to present rigorous analysis
of the original method that can be extended to these more complicated solvers.  With that in mind, we now turn to a brief
introduction of the spectral deferred correction framework, and in the process of doing so, we seek to directly compare this method
with that of the Picard integral formulation of a numerical ODE solver.

\subsection{Picard Iteration and the SDC Framework}

We begin by giving a brief description of classical SDC methods.  In doing so, we explain
the differences between SDC and that of Picard iteration, which defines the cornerstone of the present work.

Classical SDC solvers are designed to solve initial value problems of the form
\begin{align} \label{eqn:ODE}
    y' = \frac{dy}{dt} &= f(y), \quad t>0, \quad y(0) = y_0,
\end{align}
where $y$ can be taken to be a vector of unkowns.
The solution $y(t)$ can be expressed as an integral through formal integration:
\begin{align} \label{eqn:IF}
	y(t) = y_0 + \int_0^t f(y(s))ds, \quad t>0.
\end{align}
In this work, we assume that $f$ is Lipshitz continuous.  That is, we assume
\begin{equation}
    | f( z ) - f( w ) | \leq L | z - w |,
\end{equation}
for some constant $L\geq 0$ and all $z,w \in \R$.
This is sufficient to guarantee existence and uniqueness for solutions of IVP \eqref{eqn:ODE}, and produce rigorous numerical error bounds for SDC methods.

Consider a set of $M$ quadrature points
$0\leq \xi_1< \ldots < \xi_M \leq 1$
that partition the unit interval into a total of $N$
disjoint subintervals, defined by
\begin{equation*}
N = \begin{cases} 
	M-1 & \text{if both endpoints are used}, \\
	M & \text{if only one endpoint is used},  \\
	M+1 & \text{if neither endpoint is used}.
\end{cases}
\end{equation*}
We make this choice because a given quadrature rule may or may not include the
endpoints of the interval,
%
%
and this convention allows us to study Gaussian quadrature rules, uniformly spaced quadrature rules, Radau II quadrature rules
and others all within the same context.  With that in mind, we define
the right endpoints $\xi^R_n$, for $n=0,1,\dots N-1$, of each the $N$ subintervals as
\begin{equation*}
	\xi^R_n =
\begin{cases}
\xi_{n+1} & \text{if the left endpoint is included}, \\
\xi_{n} & \text{if the left endpoint not included},
\end{cases}
\end{equation*}
and $\xi^R_0 = 0, \xi^R_N = 1$ for the two boundary edge cases.  
Next, we define quadrature weights by 
\begin{equation}
	\label{eqn:weights}
	w_{n,m} = \int_{\xi^R_{n-1}}^{\xi^R_{n}} \ell_{m}(x)dx, \quad n = 1,2,\ldots N, \quad m = 1, 2, \ldots M,
\end{equation}
where $\ell_m(x)$ is the Lagrange interpolating polynomial of degree at most $M-1$ corresponding to the 
quadrature point $\xi_m$:
\begin{equation}
	\ell_m(x) = \frac{1}{c_m}\prod_{k=1,k\neq m}^M(x-\xi_k), \qquad c_m = \prod_{k=1,k\neq m}^M(\xi_m-\xi_k).
\end{equation}
Once these weights are obtained, approximate integral solutions,
say 
$\eta_m \approx y(\xi^R_m h)$ for $h>0$ and $m=0,1,\dots,N$, can be formed via
\begin{equation}
	\label{eqn:IFQ}
	{\bf Fully\, Implicit\, Collocation:} \quad
	\eta_n = \eta_{n-1} + h\sum_{m=1}^M w_{n,m} f(\eta_m), 
	\quad n=1,2,\dots, N,
\end{equation}
whereas the exact solution $y_m := y(\xi^R_m h)$ satisfies the exact integral
\begin{equation}
\label{eqn:ODE-exact}
    y_n = y_{n-1} + \int_{t_{n-1}}^{t_{n}} f( y(t) )\, dt, \quad t_n = \xi^R_n h, \quad n=1,2,\ldots N.
\end{equation}
By convention, $\eta_0 := y_0$ is known to high-order (because it comes from the previous time step), and $\eta_N \approx y( h )$
constitutes one ``full" time step. 
Since each substep uses information from all
substeps to construct the right-hand side, the solution is higher order, but
also requires the solution of a nonlinear system of $M$ unknowns (one for each quadrature point) at each time step.
Although this integrator has some very nice properties (e.g., it can be made to be symplectic and $L$-stable for suitably chosen
quadrature points),
it is not typically used in practice given the additional storage requirements and the larger matrices that need to be inverted
for each time step.
This is particularly relevant when it is used as the base solver for a partial differential equation, but even these bounds are being
explored as a viable option for PDE solvers such as the discontinuous Galerkin method \cite{PaPer17}.

In place of the fully implicit collocation method, Picard iteration (with numerical quadrature) defines a solver by iterating on a current solution $\eta^{[p]}_n$, $p \in \Z_{\geq 0}$ and
then creates a better approximation through
\begin{equation}
	\label{eqn:picard-explicit}
{\bf Picard\, Iteration:}	\quad \eta^{[p+1]}_n = \eta^{[p+1]}_{n-1} + h \sum_{m=1}^M w_{n,m} f(\eta^{[p]}_m), \quad n=1,2,\ldots N.
\end{equation}
Note that the current value $\eta^{[p+1]}_0 := \eta_0 \approx y_0$ is a known value that is equal to the exact solution up to high-order.
While this solver picks up a single order of accuracy with each correction, it has the unfortunate consequence of having a finite region
of absolute stability.  

The explicit spectral deferred correction framework is
\begin{equation}
    \label{eqn:EXP_p}
{\bf Exp.\, SDC:}  \quad  \eta_n^{[p+1]} = \eta_{n-1}^{[p+1]} + h_n
    \left[f(\eta_{n-1}^{[p+1]})-f(\eta_{n-1}^{[p]})\right] + h \sum_{m=1}^M w_{n,m}f(\eta_m^{[p]}),
\end{equation}
where $h_n = (\xi^R_n - \xi^R_{n-1}) h$ is the length of the $n^{th}$ sub-interval.  This solver also has a finite region of
absolute stability.

\begin{rmk}
Although traditional SDC methods were originally cast as a method that
corrects a provisional solution by solving an error equation, some modern descriptions of the same solver identify Eqn.\ \eqref{eqn:EXP_p} as
the base solver, which has the added benefit of pointing out a solid link between SDC methods
and that of iterative Picard integral equation solvers.
%
\end{rmk}

In order to construct methods that have more favorable regions of absolute stability for stiff problems,
the implicit SDC framework
exacts multiple backward Euler time steps through each iteration with
\begin{equation}
    \label{eqn:IDC_p}
    {\bf Imp.\, SDC:} \quad \eta_n^{[p+1]} = \eta_{n-1}^{[p+1]} + h_n
    \left[f(\eta_n^{[p+1]})-f(\eta_n^{[p]})\right] + h \sum_{m=1}^M w_{n,m}f(\eta_m^{[p]}).
\end{equation}
Note that this framework allows for implicit and high-order solutions to be constructed with 
greater computational efficiency
when compared to the fully implicit collocation solver defined in Eqn.\,\eqref{eqn:IFQ} because
smaller systems need to be inverted in order to take a single time step.

\begin{rmk}
It has been noted that the scaling in front of the $h_n$ term does not have impact on the order of accuracy \cite{XiaXuShu07}.  It is our aim 
with this work to solidify that claim with rigorous numerical bounds, which we do for both the explicit and implicit solvers.
\end{rmk}

Before doing so, we point out an aside that is in common with all SDC solvers.
\begin{rmk}
\label{rmk:conv}
If $\lim_{p\to\infty} \eta^{[p]}_n = \eta_n$ converges, then solutions to Eqns.\ \eqref{eqn:picard-explicit}, \eqref{eqn:EXP_p},
and \eqref{eqn:IDC_p}
converge to that of the fully implicit collocation method defined in Eqn.\ \eqref{eqn:IFQ}.  
\end{rmk}

While some SDC methods work with a fixed number of iterates in order to obtain
a desired order of accuracy, there are many examples in the literature where convergence of the SDC
iterations to the fully implicit scheme is considered.  For example, the work
in \cite{Weiser15} is wholly concerned with this convergence, and not the
accuracy of the underlying method for fixed iterations.  Moreover, for stiff
problems, it is  well understood that using a fixed number of iterations can
lead to order reduction which negates the advantage of using SDC in the first
place.  In addition, the multilevel spectral deferred correction (MLSDC) methods also are typically iterated to a residual tolerance,
since one cannot be sure that coarse level sweeps will provide enough increase
in accuracy (or decrease in the residual) \cite{SpeRuEmMiBoKr15}.

One key advantage of iterating an SDC method to convergence is that when this is done, the method inherits well known and desirable properties that the fully implicit
collocation method enjoys.  For example, Kuntzmann \cite{Kuntzmann61} and Butcher \cite{Butcher64} separately 
point out that if a total of $M$ Gaussian quadrature points are used, then the fully implicit collocation method will have superconvergence order $\BigOh(
h^{2M} )$.  (For more details, we refer the interested reader to the excellent tomes of Hairer, Wanner et al \cite{HaNoWa93,HaWa96,HaLuWa10}.
For example, see Section II.7 of \cite{HaNoWa93},
Theorem 5.2 from \cite{HaWa96}, or Theorem 1.5 from \cite{HaLuWa10}.)  In general, the maximum order of accuracy for the underlying
solver with $M$ quadrature points is $\BigOh(h^{2M})$ if they are Gauss-Legendre points, $\BigOh( h^{2M-1} )$ for the RadauIIA points, and
$\BigOh( h^{2M-2} )$ for Gauss-Lobatto points.  (The local truncation error is one order higher.)  Uniform points have $\BigOh(h^M)$ order of convergence
if $M$ is even, and $\BigOh( h^{M+1} )$ if $M$ is odd.  The extra pickup in the order of accuracy is due to symmetry of the quadrature rule.  (For example, $M=1$ 
points reproduces the so-called ``midpoint" rule,
$M=3$ reproduces Simpson's rule, and $M=5$ yields Boole's rule, each of which pick up an extra order of accuracy.)

\subsection{An Outline of the Present Work}
%
%
Despite the increasing popularity of spectral deferred correction solvers, very little work has been
performed on convergence results for this large class of methods.  The results that are currently in the literature
\cite{HagZhou06,ChriOngQiu09,ChriOngQiu10,HansenStrain11,TangXieYin13} typically 
proceed via induction on the current
order of the approximate solution, and they all hinge on solving the
\emph{error equation}, wherein the same low-order solver is applied and then a defect, or correction,
is added back into the current solution in order to increase its overall order of accuracy.  
In other recent work \cite{QuBraCheHuaKre16}, the authors consider SDC methods as fixed point iterations on a Neumann series expansion.
There, the low-order method is viewed as an efficient preconditioner (in numerical linear algebra language), and the SDC iterations are thought of as simplified Newton iterations.
Additionally, the work in \cite{Weiser15} makes use of linear algebra techniques in order to optimize coefficients so that the method converges faster to the collocation solution for stiff problems.

In this work, we do not require the use of the error equation, nor do we work with any sort of defect 
such as that defined in \cite{Zadu76}, rather we instead focus on the Picard integral underpinnings inherent to all SDC methods.
Our work solely uses fundamental numerical analysis tools: error estimates for numerical interpolation and integration.
While these tools do rely on quadrature rules, our proofs are generic enough to accommodate any set of quadrature points,
which are an ongoing discussion in terms of how to construct base solvers.
%


In this work, we prove rigorous error bounds for both implicit and explicit SDC methods, and in doing so, we expect the reader
will find that these methods can be thought of as being built upon classical Picard iteration.  Our results are applicable for general
quadrature rules, but unlike the findings
found in \cite{TangXieYin13}, where convergence is proven using the error equation, 
our work relies on the fundamental mechanics behind why the solver works.  That is, we point out that the primary contributor to the order of accuracy of the solver
lies within the integral of the residual, and not necessarily the application of any base solver to an error equation.

Indeed, our proofs follow in similar manner to that required to
prove the Picard--Lindel\"{o}f theorem, but our proofs 
take into account numerical quadrature errors and 
do not rely on exact integration of the right hand side function $f(y)$.
The primary differences between our proofs and that of the Picard-Lindel\"{o}f theorem are the following:
\begin{itemize}
    \item Spectral deferred correction methods require the use of \emph{numerical quadrature} to
        approximate the integrals presented in the Picard-Lindel\"{o}f theorem.  Our error estimates take into
        account any errors resulting from quadrature rules.
    \item Each correction step in the \emph{implicit scheme} defined in  \eqref{eqn:IDC_p} requires a nonlinear inversion, 
	whereas the Picard-Lindel\"{o}f theorem is typically proven using exact integration.
\end{itemize}


There are two main results in this work, one for explicit SDC methods, and one for implicit SDC methods.  
These are both found as corollaries to a single theorem on semi-implicit SDC.  In each case, we
produce rigorous error bounds that are applicable for generic quadrature rules.  Furthermore, we find that there are
a total of three sources of error that SDC methods carry: i) the error from the previous time step, ii) the error from the previous iterate, and 
iii) the error from the quadrature rule being used.

The outline of this paper is as follows.  In Section \ref{section:prelim}, we present some necessary lemmas concerning
error estimates for integrals of interpolants as well as some error estimates for sequences of inequalities that show up in our proofs.
In Section \ref{sec:convergence-results} we present a convergence proof for the more general case of a semi-implicit SDC solver,
and then immediately point out two corollaries that prove implicit and explicit SDC methods converge.
In Section \ref{sec:trap}, we present results for an SDC method that makes use of a higher order base solver, the trapezoidal rule.
In Section \ref{section:num-results}
we present some numerical results, where we compare explicit SDC methods with Picard iterative methods, 
we investigate modified implicit SDC methods, and we experiment with
different semi-implicit formulations of SDC methods.  Error estimates for all of these variants come from direct extensions of the proofs found in this work.
Finally, some conclusions and suggestions for future work are drawn up in Section \ref{section:conclusions}.

\section{Preliminaries}
\label{section:prelim}

We now point out a couple of important tools that we use to show that SDC solvers converge.  
Our aim is to focus on a single time step.  Without loss of generality, from here on out
we will focus on constructing a solution over the interval $[0,h]$, where $h$ is the time step size and
we will assume that $\eta_0 \approx y_0$ is a high-order approximation to the exact solution. 

\subsection{Error estimates for integrals of interpolants}

If ${\boldsymbol \eta} = \left( \eta_1, \eta_2, \cdots, \eta_M \right)$ is a set of discrete values
and $t \in [0, h]$ is a time interval we are interested in studying,
we define the \emph{interpolation operator} $I$ to be the projection onto the space of polynomials of degree at most $M-1$ via
\begin{equation}
\label{eqn:interpolation}
    I[f( {\boldsymbol \eta} )]( t ) := \sum_{m=1}^M f( \eta_m ) l_m( t / h ),
    \quad
    f( {\boldsymbol \eta} ) := \left( f( \eta_1 ), f( \eta_2), \cdots, f( \eta_M ) \right).
\end{equation}
Note that this produces the integration identity
\begin{equation}
    \int_{t_{n-1}}^{t_n} I[f( {\boldsymbol \eta })]( t )\, dt = h \sum_{m=1}^M w_{n,m}f(\eta_m )
\end{equation}
after integrating \eqref{eqn:interpolation} over a subinterval $[t_{n-1}, t_n] := [h \xi_{n-1}^R, h \xi_n^R]$,
and the weights are defined as in Eqn.\,\eqref{eqn:weights}.


Convergence results for both the explicit and the implicit SDC method (as well as Picard iteration)
require the use of the following lemma.  

\begin{lemma}
\label{lem:lipshitz-integration}
Suppose that $f\circ y \in C^{M}( [0,h] )$, $\norm{ \frac{d^M}{dt^M} \left( f\circ y \right) }_\infty \leq F$, and 
that $f$ is Lipschitz continuous with Lipshitz constant $L$.  Then
we have the estimate
\begin{equation}
\label{eqn:pic-lemma}
\begin{aligned}
	\left| \int_{t_{n-1}}^{t_n} I[f( \boldsymbol \eta ) ]( t ) - f( y (t) ) \, dt \right| \leq 
	h \norm{ {\boldsymbol \eta} - {\boldsymbol y}} W_{n} L 
	+ \frac{ F }{ M! } h^{M+1}.
\end{aligned}
\end{equation}
where the discrete norm is defined by
\begin{equation}
	\norm{ {\boldsymbol e} } := \max_{1 \leq n \leq M} \left | e_n \right |,
    \quad
    {\boldsymbol e} = \left( e_1, e_2, \cdots, e_M \right),
\end{equation}
and the constant $W_n$ is defined by
\begin{equation}
\label{eqn:Wn}
	W_n := \sum_{m=1}^M \int_{\xi^R_{n-1}}^{ \xi^R_n} | l_m(\xi) | d\xi.
\end{equation}
For a fixed quadrature rule, this constant is finite and independent of the function.
\end{lemma}

\begin{proof}
Add and subtract the Lagrange interpolant $I[ f( {\boldsymbol y} ) ](t)$ for $f\circ y$ inside the left hand side of 
Eqn.\ \eqref{eqn:pic-lemma} and apply the triangle inequality:

\begin{equation}
\label{eqn:lem1-split}
\begin{aligned}
	\left| \int_{t_{n-1}}^{t_n} I[f( {\boldsymbol \eta } ) ]( t ) - f( y(t) ) \, dt \right| &\leq 
	\left| \int_{t_{n-1}}^{t_n} I[f( {\boldsymbol \eta } ) ]( t ) - I[ f( {\boldsymbol y} ) ](t)\, dt \right| \\
	&+
	\left| \int_{t_{n-1}}^{t_n} I[ f( {\boldsymbol y} ) ](t) - f( y(t) ) \, dt \right|.
\end{aligned}
\end{equation}
An estimate for the first of these two terms follows by linearity of the interpolation operator:
\begin{equation}
\begin{aligned}
	\left| \int_{t_{n-1}}^{t_n} I[f( {\boldsymbol \eta } ) ]( t ) - I[ f( {\boldsymbol y} ) ](t)\, dt \right| 
	&=
	\left| h \sum_{m=1}^M \omega_{n,m} \left( f( \eta_m ) - f( y_m ) \right) \right| \\
	&\leq h \sum_{m=1}^M | \omega_{n,m} | \left| f( \eta_m ) - f( y_m ) \right| \\
	&\leq h L \sum_{m=1}^M | \omega_{n,m} | \left| \eta_m - y_m \right| \\
	&\leq h L \norm{ {\boldsymbol \eta } - {\boldsymbol y} } \sum_{m=1}^M | \omega_{n,m} |.
\end{aligned}
\end{equation}
The quadrature weights in this estimate are be bounded above by
$| \omega_{n,m} |
	\leq \int_{\xi^R_{n-1}}^{\xi^R_{n}} \left| \ell_{m}(\xi) \right| d\xi$
and then summed over all $m$ to produce the constant $W_n$.

The second of the two integrals in \eqref{eqn:lem1-split}
is a function solely of the smoothness of $f$ and the choice of the quadrature
rule.  That is, classical interpolation error estimates result in a bound on
the $M^{th}$ derivative of $f\circ y$ through
a single point $z(t) \in [0, h]$ that yields
\begin{equation}
\label{eqn:lagrange-error}
	\left| I[ f( {\boldsymbol y} ) ](t) - f( y(t) ) \right| = \left| \frac{ (f \circ y)^{(M)} \left( z(t) \right) }{ M! } \prod_{m=1}^M (t - t_m) dt \right|
	\leq \frac{F}{M!} \prod_{m=1}^M | t - t_m|.
\end{equation}
Because $|t-t_m| \leq h$ for each $m$, the result follows after integration.
%
%
\end{proof}

We stop to point out that due to the Runge-phenomenon, the coefficient $W_n$ defined in Equation \eqref{eqn:Wn} can become quite large
if a large number of quadrature points are chosen for constructing the polynomial interpolants required for the SDC method.
In Table \ref{tab:Wn}, we demonstrate a few sample values when uniform, Chebyshev, Gauss-Legendre, Gauss-Radau, and Gauss-Lobatto quadrature
nodes are used to construct
the polynomial interpolants.  Uniform quadrature points tend to start performing
quite poorly in the teens; however even a small amount of points, say five or six, produces a high-order numerical method compared to other
ODE solvers because in this regime the error constant is reasonable.  
The selection of quadrature points that minimizes this portion of the error constant is the Gauss-Lobatto nodes, but because convergence is found through refinement in $h$ rather than $p$, any of these points will produce a method that converges, provided the exact solution has a suitable degree of regularity.

%

\begin{table}
{\footnotesize
\caption{
Maximum size of the Lagrange polynomials 
$
\max_{1 \leq n \leq M} \max_{\xi \in [0,1] } | \ell_n(\xi) |
$
for different quadrature points.
The Gauss-Legendre, Gauss-Radau, and Gauss-Lobatto quadrature rules with $M$ points have degrees of precision
$2M+1, 2M$, and $2M-1$, respectively.
\label{tab:Wn}}
\begin{tabular}{|r||c|c|c|c|c|}
\hline
& \multicolumn{5}{|c|}{{\bf Type of Quadrature Points}} \\ \hline 
\bf{M} & \bf{Uniform} & \bf{Chebyshev} & \bf{Legendre} & \bf{Gauss-Radau} & \bf{Gauss-Lobatto} \\\hline 
\hline
$ 2$ & $1.000$ & $1.207$ & $1.366$ & $1.500$ & $1.000$ \\ 
\hline
$ 3$ & $1.000$ & $1.244$ & $1.479$ & $1.558$ & $1.000$ \\ 
\hline
$ 4$ & $1.056$ & $1.257$ & $1.527$ & $1.578$ & $1.000$ \\ 
\hline
$ 5$ & $1.152$ & $1.263$ & $1.551$ & $1.586$ & $1.000$ \\ 
\hline
$ 6$ & $1.257$ & $1.266$ & $1.566$ & $1.591$ & $1.000$ \\ 
\hline
$ 7$ & $1.362$ & $1.268$ & $1.575$ & $1.594$ & $1.000$ \\ 
\hline
$ 8$ & $1.663$ & $1.269$ & $1.581$ & $1.596$ & $1.000$ \\ 
\hline
$ 9$ & $2.550$ & $1.270$ & $1.585$ & $1.597$ & $1.000$ \\ 
\hline
$ 10$ & $4.028$ & $1.271$ & $1.588$ & $1.598$ & $1.000$ \\ 
\hline
$ 11$ & $6.506$ & $1.271$ & $1.590$ & $1.599$ & $1.000$ \\ 
\hline
$ 12$ & $10.963$ & $1.271$ & $1.592$ & $1.599$ & $1.000$ \\ 
\hline
$ 13$ & $18.340$ & $1.272$ & $1.594$ & $1.600$ & $1.000$ \\ 
\hline
$ 14$ & $32.060$ & $1.272$ & $1.595$ & $1.600$ & $1.000$ \\ 
\hline
$ 15$ & $54.998$ & $1.272$ & $1.596$ & $1.600$ & $1.000$ \\ 
\hline
$ 16$ & $98.531$ & $1.272$ & $1.596$ & $1.600$ & $1.000$ \\ 
\hline
$ 17$ & $172.176$ & $1.272$ & $1.597$ & $1.601$ & $1.000$ \\ 
\hline
$ 18$ & $313.675$ & $1.272$ & $1.597$ & $1.601$ & $1.000$ \\ 
\hline
$ 19$ & $556.491$ & $1.273$ & $1.598$ & $1.601$ & $1.000$ \\ 
\hline
$ 20$ & $1026.313$ & $1.273$ & $1.598$ & $1.601$ & $1.000$ \\ 
\hline
$ 30$ & $496210.554$ & $1.273$ & $1.600$ & $1.602$ & $1.000$ \\ 
\hline
$ 50$ & $208948162475.383$ & $1.273$ & $1.601$ & $1.602$ & $1.000$ \\ 
\hline
\hline 
\end{tabular} }
\end{table}


\subsection{Error estimates for sequences of inequalities}

Finally, we require a second Lemma as well as a simple Corollary.  Both of
these are stated in \cite{Farago13} and their proofs are elementary.
\begin{lemma}
\label{lem:recursive-errors}
If $\{ a_n \}_{n \in \Z_{\geq 0}}$ is a sequence that satisfies 
$|a_n| \leq A |a_{n-1}| + B$ with $A \neq 1$, then 
\begin{equation}
\label{eqn:lem1}
    |a_n| \leq A^n |a_0| + \frac{ A^n - 1 }{ A - 1 } B.
\end{equation}
\end{lemma}
\begin{proof}
    Recursively apply the inequality and sum the remaining finite geometric series.
\end{proof}

\begin{cor}
\label{cor:geo-estimate}
If $A > 1$ and $\{ a_n \}_{n \in \Z}$ is a sequence that satisfies
$|a_n| \leq A |a_{n-1}| + B$,
then
\begin{equation}
    |a_n| \leq A^n |a_0| + n A^{n-1} B
\end{equation}
for every $n$.
\end{cor}
\begin{proof}
By Lemma \ref{lem:recursive-errors}, the sequence satisfies
Eqn.\,\eqref{eqn:lem1}.  We estimate the (finite) geometric series by
\begin{equation}
    \frac{ A^n - 1 }{ A - 1 } = 1 + A + \cdots A^{n-1} 
    \leq n A^{n-1}
\end{equation}
because there are a total $n$ terms and each $A^{l} \leq A^{n-1}$ for $l=0,1,\dots, n-1$.
\end{proof}
With these preliminaries out of the way, we are now ready to state and prove our main result.

\section{Convergence Results}
\label{sec:convergence-results}

In place of separately proving explicit and implicit results for \eqref{eqn:ODE}, we instead consider an umbrella class of ODEs, defined
through a semi-implicit formulation:
\begin{equation}
\label{eqn:ODE-lip}
y' = f(y), \quad f(y) = f_I(y) + f_E(y), \quad y(0) = y_0,
\end{equation}
where $f_I$ is to be treated implicitly, and $f_E$ is to be treated explicitly.
We assume that both $f_I$ and $f_E$ have Lipshitz constants $L_I$ and $L_E$, respectively.
In turn, this implies that $f$ has a Lipshitz constant of $L := L_I + L_E$.
In the case where $f_I \equiv 0$, we set $L_I = 0$,
and in the case where $f_E \equiv 0$, we set $L_E = 0$.

The classical semi-implicit SDC (SISDC) method
for \eqref{eqn:ODE-lip}
begins with a \emph{provisional solution}, or initial guess 
$\eta^{[0]}_n \approx y(\xi_n h)$, that is typically defined with
\begin{equation}
    \label{eqn:SISDC_0}
    \eta^{[0]}_n = \eta^{[0]}_{n-1} + h_n f_I(\eta^{[0]}_n) + h_n f_E( \eta^{[0]}_{n-1} ), \quad n=1,2,\ldots N,
\end{equation}
where $h_n = (\xi^R_n - \xi^R_{n-1}) h$.  This yield a first-order implicit-explicit (IMEX) predictor for the solution based upon a forward-backward Euler
method.  Our numerical (and analytical) results indicate the ``predictor'' step has little bearing on the overall order of accuracy of the
solver.  For example, it is possible to hold the solution constant for the initial iteration and still obtain high-order accuracy, albeit
with one additional iteration.


The classical SISDC method \cite{Minion03}  iterates on the provisional solution through 
\begin{equation}
\label{eqn:sisdc-classical}
\begin{aligned}
{\bf SISDC:} \quad 
    \eta_n^{[p+1]} &= \eta_{n-1}^{[p+1]} + h_n
    \left[  f_I(\eta_{n}^{[p+1]})-f_I(\eta_{n}^{[p]})\right] \\
    &+ h_n 
    \left[  f_E(\eta_{n-1}^{[p+1]})-f_E(\eta_{n-1}^{[p]})\right] + 
    h \sum_{m=1}^M w_{n,m}f(\eta_m^{[p]}),
\end{aligned}
\end{equation}
where $\eta^{[p+1]}_0 = \eta_0$ is a known quantity.  
This value is typically taken to be the result from the previous time step,
and we assume that it is known to high-order accuracy.
Our focus is on the local truncation error, in which case we assume that the error at time zero is non-zero.
That is, we assume $e_0 = \eta_0 - y_0 \neq 0$.
Once the single step error is established, a global error can be directly found using textbook techniques.
In the event where $f_I \equiv 0$, we end up with the explicit SDC 
method defined in \eqref{eqn:EXP_p},
and when $f_E \equiv 0$, we end up the implicit SDC defined in 
\eqref{eqn:IDC_p}.

We repeat that the collocation method
defined in \eqref{eqn:IFQ} requires simultaneously solving for each
$\eta_n$ and is clearly more expensive than multiple applications of the backward
Euler method found in \eqref{eqn:sisdc-classical}, either on part of or the entire right hand side.  In the event where $f_I \equiv 0$,
then the method should be less expensive to run for a single time step, but the regions of absolute stability 
suffer \cite{LaytonMinion05,LaytonMinion07}.
We also repeat that provided that if $\boldsymbol \eta^{[p]}$ converge as $p \to \infty$, then
Eqn.\,\eqref{eqn:sisdc-classical} defines a solution to Eqn.\,\eqref{eqn:IFQ}.
Proving which initial guesses converge to the fully implicit solver
is beyond the scope of this work.  Currently, our aim is to
show that each correction step in the SDC framework picks up
at least a single order of accuracy to the order predetermined by the quadrature rule.

\subsection{Statement of the Main Result}

\begin{theorem} 
\label{thm:error-sisdc}
The errors for a single step of the semi-implicit SDC method satisfy
\begin{equation}
\label{eqn:theorem-sisdc}
        |e^{[p+1]}_n| \leq e^{N h ( 2 L_I + L_E ) } |e_{0}| + C_1 h \norm{ {\bf e}^{ [p] } } + C_2 h^{M+1},
\end{equation}
provided $h L_I < 1 / 2$, and where $N$ is the number of intervals under consideration, $L := L_I + L_E$ is the Lipschitz
constant of $f$,
\[
C_1 = 2 N e^{Nh (2 L_I + L_E) } W, \quad \text{and} \quad
C_2 = 2 N e^{Nh( 2 L_I + L_E) } \frac{F}{M!}
\]
are constants that depend only on $f$, the exact solution $y$, and the
selection of quadrature points.

\end{theorem}

In subsection \ref{subsec:cor} we point out two corollaries to this result, one for implicit and one for explicit SDC,
but before proving this theorem, we stop to point out an important observation that is applicable to any of the aforementioned methods.

\begin{rmk}
The statement of this theorem highlights that there are a total of three sources of error that SDC methods admit,
which are ordered by appearance in the right hand side of 
\eqref{eqn:theorem-sisdc}:
\begin{enumerate}
\item the error at the current time step: $e_0 = \eta_0 - y_0$,
\item the error from the previous iterate (or predictor): ${\boldsymbol e}^{[p]} = {\boldsymbol \eta}^{[p]} - {\bf y}$, and
\item the number of quadrature points, $M$.
\end{enumerate}
\end{rmk}
%
The most important take-away is that 
{\bf because the error from the previous
iterate, $\norm{ {\bf e}^{ [p] } }$, gets multiplied by a factor of $h$, the
error gets improved by one order 
of accuracy with each correction.} 
Of course this order reaches a maximum order based upon the number of the quadrature points chosen,
which can be seen in the third source of error.  This can be improved by
selecting quadrature points with superconvergence properties such as the
Gaussian or Gauss-Lobatto quadrature points.  
Finally, please note that we make no comment
about how the ``previous" function values were found.  This is intentionally done so, because we would like to focus our attention
on the impact of what a single correction does the solution.  In doing so, this permits the analysis to apply to parallel implementations of SDC methods
where synchronizations between different correctors (threads) are seldom seen \cite{ChrMacOng10,EmmettMinion12}.

\subsection{Proof of the Main Result}
\label{subsec:proof}

\begin{proof}
We subtract the exact equation \eqref{eqn:ODE-exact} from \eqref{eqn:sisdc-classical} and find that the
discrete error evolution equation is
\begin{equation}
\label{eqn:error-imp}
\begin{aligned}
    e^{[p+1]}_n = e^{[p+1]}_{n-1} + h_n 
    \left[f_I(\eta_{n}^{[p+1]})-f_I(\eta_{n}^{[p]})\right] +
    h_n \left[f_E(\eta_{n-1}^{[p+1]})-f_E(\eta_{n-1}^{[p]})\right] \\ +
    \int_{t_{n-1}}^{t_n} I[f( {\boldsymbol \eta}^{[p]} ) ]( t ) - f( y(t) ) \, dt.
\end{aligned}
\end{equation}
%
The last term in this summand can be estimated by 
appealing to Lemma \ref{lem:lipshitz-integration} and observing
\begin{equation}
    |I_n| :=  \left| \int_{t_{n-1}}^{t_n} I[f( {\boldsymbol \eta}^{[p]} ) ]( t ) - f( y(t) ) \, dt \right| 
    \leq 
	h \norm{ {\boldsymbol e}^{[p]} } W_{n} L 
	+ \frac{ F }{ M! } h^{M+1}.
\end{equation}
We estimate the other terms by making use of their respective Lipshitz constants:
\begin{equation}
\begin{aligned}
|e^{[p+1]}_n| &\leq |e^{[p+1]}_{n-1}| +
    h_n \left| f_I(\eta_{n}^{[p+1]})-f_I(\eta_{n}^{[p]})\right| +
    h_n \left| f_E(\eta_{n-1}^{[p+1]})-f_E(\eta_{n-1}^{[p]})\right| +
    \left| I_n \right|
\\ &\leq |e^{[p+1]}_{n-1}| +
    h L_I \left| \eta_{n}^{[p+1]} - \eta_{n}^{[p]}  \right| +
    h L_E \left| \eta_{n-1}^{[p+1]} - \eta_{n-1}^{[p]}  \right| +
         \left| I_n \right|
\\ &\leq |e^{[p+1]}_{n-1}| 
        + h L_I \left( | e_{n}^{[p+1]} | + | e_{n}^{[p  ]} | \right)
        + h L_E \left( | e_{n-1}^{[p+1]} | + | e_{n-1}^{[p  ]} | \right) +
         \left| I_n \right|.
\end{aligned}
\end{equation}
The third line follows from the second by adding and subtracting $y_n$ to the inside each of the absolute values containing
$|\eta_{n}^{[p+1]} - \eta_{n}^{[p]}|$ and $|\eta_{n-1}^{[p+1]} - \eta_{n-1}^{[p]}|$.
Note that we
also make use of the fact that $h_n \leq h$,
although this too can be relaxed.

We continue
by subtracting $h L_I | e_n^{[p+1]} |$ from both sides, dividing by
$1-h L_I > 0$, recognizing that 
$h_n < h$, and collecting the remaining terms involving the ``explicit" portions:
\begin{equation}
\begin{split}
\label{eqn:error-estimate-imp}
    |e^{[p+1]}_n| &\leq 
    \frac{1}{ 1 - h L_I } \left[ \left( 1 + h L_E \right) |e^{[p+1]}_{n-1}| 
        + h L_I \left| e^{[p]}_{n} \right| 
        + h L_E \left| e^{[p]}_{n-1} \right|
        + |I_n| \right] \\
	&\leq
	    \frac{1}{ 1 - h L_I } \left[ \left( 1 + h L_E \right) |e^{[p+1]}_{n-1}| 
        + h L \norm{ {\boldsymbol e}^{[p]} }
        + |I_n| \right] \\
    &\leq
    	\frac{1}{ 1 - h L_I } \left[ \left( 1 + h L_E \right) |e^{[p+1]}_{n-1}| 
    		+ h L (1 +  W_n ) \norm{ {\boldsymbol e}^{[p]} } + \frac{F}{M!} h^{M+1} \right] \\
    &\leq
    \frac{1  + h L_E }{ 1 - h L_I }  |e^{[p+1]}_{n-1}| + 
    \frac{1}{ 1 - h L_I } \left[
    h W \norm{ e^{[p]} } + \frac{F}{M!} h^{M+1} \right],
\end{split}
\end{equation}
where we define $W := \max_{1 \leq n \leq N} \left( 1 + W_n L \right)$.

We make use of two separate estimates for $1/(1-hL_I)$ to estimate the two terms
found in the right hand side of \eqref{eqn:error-estimate-imp}. For the first term, we expand
the geometric series and keep the first two terms:
\begin{equation}
    \frac{1}{ 1 - h L_I } = 1 + (hL_I) + (hL_I)^2 + \cdots = 1 + (hL_I) + (hL_I)^2 \frac{1}{1-hL_I}.
\end{equation}
This is valid because $hL_I < 1$.  Additionally, $hL_I < 1/2$, and therefore
\begin{equation}
h L_I < 1 - h L_I \implies
	\frac{(hL)^2}{1-h L_I } <  h L_I.
\end{equation}
Together, these estimates imply that the first term can be estimated with
\begin{equation}
\label{eqn:geo-estimate}
    \frac{1}{ 1 - h L_I } \leq 1 + 2 h L_I \leq e^{2hL_I}.
\end{equation}
For the second term, we have $1/(1-hL_I) \leq 2$ for all $hL_I \in [0,1/2]$.  This
leads us to observe that
\begin{equation}
    |e^{[p+1]}_n| \leq 
    e^{2hL_I} \left( 1 + h L_E \right) |e^{[p+1]}_{n-1}| + 2 \left( h W \norm{ {\bf e}^{[p]} }
    + \frac{F}{M!} h^{M+1} \right).
\end{equation}

Next, we appeal to Corollary \ref{cor:geo-estimate}
and make use of $A = e^{2hL_I} \left( 1 + h L_E \right) > 1$ and 
$B = 2 \left( h W \norm{ {\bf e}^{[p]} } 
    + \frac{F}{M!} h^{M+1}
\right)$
to conclude that
\begin{equation}
	|e^{[p+1]}_n| \leq e^{ 2 h n L_I } \left( 1 + h L_E \right)^n |e_{0}| +n 
	e^{2h (n-1) L_I} \left( 1 + h L_E \right)^{n-1}
	2 \left( h W \norm{ {\bf e}^{ [p] } } + \frac{ F }{ M! } h^{M+1} \right).
\end{equation}
Since $1 + h L_E \leq e^{h L_E}$ and $n \leq N$, we have the desired result.
%
%
\end{proof}

\subsection{Corollaries of Main Result: Implicit and Explicit Error Estimates}
\label{subsec:cor} 

With the general case proven in Theorem \ref{thm:error-sisdc}, we find results
for both implicit, as well as explicit SDC solvers.
An immediate corollary to Theorem \ref{thm:error-sisdc} can be found by setting $f_E \equiv 0$, in which case $L_I$ becomes the Lipshitz constant for $f$,
and the SISDC solver reduces to classical SDC with backward Euler defined in \eqref{eqn:IDC_p}.

\begin{cor} 
\label{cor:error-imp}
The errors for a single step of the implicit SDC method defined in \eqref{eqn:IDC_p} satisfy
\begin{equation}
\label{eqn:theorem-imp}
        |e^{[p+1]}_n| \leq e^{2 N h L} |e_{0}| + C_1 h \norm{ {\bf e}^{ [p] } } + C_2 h^{M+1}
\end{equation}
provided $h < 1 / (2L)$.  
The constants $C_1$ and $C_2$
depend only on the smoothness of $f$, the exact solution $y$, and the choice
of quadrature points.
\end{cor}

It is worth noting that the error estimate provided here is an \emph{asymptotic} error estimate.  That is, 
one key assumption that we have to make is that $h < 1 / (2L)$, which we do not 
have to make for the explicit case.  
Unfortunately, one key benefit of implicit solvers is that large time steps can be taken, in which case it is certainly possible
that the solver does not obey this assumption.  For these cases, a rigorous error estimate and analysis when $h > 1/ (2L)$
would make for an interesting result, which would be especially important for multiscale problems that contain large time scale separations.
This observation is beyond the scope of the present work.

A related corollary for explicit solvers with tighter error bounds can be found.  The result is the following.

\begin{cor}
\label{cor:error-exp}
The errors for a single step of the explicit SDC method defined in 
\eqref{eqn:EXP_p}
satisfy
\begin{equation}
\label{eqn:theorem-exp}
        |e^{[p+1]}_n| \leq e^{N h L} |e_{0}| + C_1 h \norm{ {\bf e}^{ [p] } } + C_2 h^{M+1},
\end{equation}
where $N$ is the number of intervals under consideration, $L$ is the Lipschitz
constant of $f$ for the ODE $y' = f(y)$, 
and $C_1$ and $C_2$ are constants that depend only on $f$, the exact solution $y$, and the selection of quadrature points.
\end{cor}

\begin{proof}
Revisit the proof of
Theorem \ref{thm:error-sisdc} and replace the error estimate for $1/(1-hL_I) \leq 2$ with 
1 instead of 2.
\end{proof}

\section{Convergence Proofs for Higher Order Base Solvers}
\label{sec:trap}

We now consider the spectral deferred correction method with the implicit trapezoidal rule as its base solver:
\begin{equation}
	\label{eqn:SDC_Trap}
    \eta_n^{[p+1]} = \eta_{n-1}^{[p+1]} + \frac{h_n}{2}
    \left[
        f(\eta_n^{[p+1]}) + f(\eta_{n-1}^{[p+1]})
       -f(\eta_n^{[p]  }) - f(\eta_{n-1}^{[p]  }) 
    \right] 
    + h \sum_{m=1}^M w_{n,m}f(\eta_m^{[p]}).
\end{equation}
What makes this method interesting is that it picks up a total of two orders of accuracy with each correction.
Note again, that in the absence of the terms that the factor $h_n/2$ multiplies, this method reduces to explicit Picard iteration, which
picks up a single additional order of accuracy with each correction.

The result we focus on is the impact of each correction step, in which case any provisional solution may be used.  
In order to retain large regions of absolute stability
reasonable methods include low-order implicit solvers such as
backward Euler, or the second-order implicit trapezoidal (Crank-Nicholson) rule:
\begin{equation}
    \label{eqn:IDC_TRAP}
    {\bf Trap.\, Rule:} \quad \eta^{[0]}_n = \eta^{[0]}_{n-1} + \frac{h_n}{2} \left( f(\eta^{[0]}_{n-1}) + f( \eta^{[0]}_n ) \right), \quad n=1,2,\ldots N.
\end{equation}
In this section, we examine the interplay between the integral over the entire time interval, and the addition of extra integral terms that allows this solver to pick up additional orders
of accuracy.

Let us define the exact value of the right hand side function as $f_n = f(y(t_n))$, the approximate value of the right hand side function
as $f_n^{[p]} = f(\eta^{[p]}_n)$, the local and global quadrature rules for integration over the subinterval $[t_{n-1}, t_n]$ as
\begin{align*}
	T_n &= \frac{h_n}{2} \left[f_{n-1}+f_{n}\right],	\qquad T_n^{[p]} = \frac{h_n}{2} \left[f^{[p]}_{n-1}+f^{[p]}_{n}\right], \\
	H_n &= h \sum_{m=1}^M \omega_{n,m}f_m,	\qquad H_n^{[p]} = h \sum_{m=1}^M \omega_{n,m}f^{[p]}_m.
\end{align*}
These definitions allow us to compactly write the SDC method with the Trapezoidal rule defined in \eqref{eqn:SDC_Trap} to read
\begin{equation}
\label{eqn:SDC_Trap-integral}
	\eta_{n}^{[p+1]} = \eta_{n-1}^{[p+1]} + \left( T_n^{[p+1]} - T_n^{[p]} \right) + H_n^{[p]}.
\end{equation}
Recall that the exact solution satisfies the integral equation \eqref{eqn:ODE-exact}, which we repeat:
\begin{equation*}
    y_n = y_{n-1} + \int_{t_{n-1}}^{t_{n}} f( y(t) )\, dt.
\end{equation*}


\begin{theorem} 
\label{thm:error-trap}
When coupled with the implicit Trapezoidal rule, the errors for a single step of the Spectral Deferred Correction method satisfy
\begin{equation}
\label{eqn:theorem-trap-rule}
        |e^{[p+1]}_n| \leq e^{ 2 N h L } |e_{0}| + 2 N e^{ 2 (N-1) h L } \left( \frac{1}{12} \norm{ \frac{d^2 E^{[p]}}{dt^2} ( \cdot ) } h^3 + \frac{F}{M!} h^{M+1} \right),
\end{equation}
provided $h L < 1$, 
$N$ is the number of intervals under consideration, $M$ is the number of points involved, $L$ is the Lipschitz
constant of $f$, $F$ is an upper bound for the $M^{th}$ derivative of $f$, and the function $E^{[p]}(t)$ is the polynomial interpolant for the error in the the approximation 
of the right hand side function during the $p^{th}$ iterate defined by
\begin{equation}
\label{eqn:error-glob-interp}
	E^{[p]}(t) := I[f({\boldsymbol \eta^{[p]} })](t) - I[f( {\boldsymbol y} ) ]( t ) = \sum_{m=1}^M \df_m^{[p]} \ell_m( t / h ),
\end{equation}
where $\df^{[p]}_m := f^{[p]}_m - f_m$ for each $m = 1, 2, \dots M$.
The norm defined in \eqref{eqn:theorem-trap-rule} is the maximum absolute value of the second derivative of $E^{[p]}$:
\begin{equation}
	\norm{ \frac{d^2 E^{[p]}}{dt^2} ( \cdot ) } := \max_{t\in [0,h]} | ( E^{[p]} )''( t ) |.
\end{equation}

\end{theorem}

\begin{proof}
We subtract the exact solution defined in Eqn.\,\eqref{eqn:ODE-exact} from the SDC method based upon the trapezoidal rule defined in Eqn.\,\eqref{eqn:SDC_Trap-integral}
to end up with
\begin{equation}
\label{eqn:trap-error}
\begin{aligned}
	e_{n}^{[p+1]} &= e_{n-1}^{[p+1]} + \left( T_n^{[p+1]} - T_n^{[p]} \right) + H_n^{[p]} - \int_{t_{n-1}}^{t_{n}} f(y(t))dt \\ 
	&= e_{n-1}^{[p+1]} + T_n^{[p+1]}-T_n +T_n-T_n^{[p]} + H_n^{[p]}-H_n + H_n- \int_{t_{n-1}}^{t_{n}} f(y(t))dt \\ 
	&= e_{n-1}^{[p+1]} + \underbrace{\left(T_n^{[p+1]}-T_n\right)}_{\bf I} +
		\underbrace{ \left(H_n^{[p]}-T_n^{[p]}+T_n-H_n\right)}_{{\bf II}}  + 
		\underbrace{ I_n }_{\bf III}, 
\end{aligned}
\end{equation}
where $I_n := H_n- \int_{t_{n-1}}^{t_{n}} f(y(t))dt$ is the difference between the high-order (discrete) integral and the exact integral of the right hand side.

We now estimate each of the three terms to the right of $e^{[p+1]}_{n-1}$ in Eqn.\,\eqref{eqn:trap-error} separately, starting with the first term:
\begin{equation}
\label{eqn:error-I}
	\left| {\bf I} \right| = \frac{h_n}{2} \left| f^{[p+1]}_{n-1} + f^{[p+1]}_{n} - f_{n-1} - f_{n} \right| \leq \frac{L h_n}{2} \left( |e^{[p+1]}_n| + |e^{[p+1]}_{n-1}| \right),
\end{equation}
which follows from using the Lipshitz continuity of $f$.  The third term can be estimated by first recognizing that
\[
	H_n := h \sum_{m=1}^M \omega_{n,m}f_m = \int_{t_{n-1}}^{t_n} I[f( {\boldsymbol y} ) ]( t ) \, dt,
\]
and then using Eqn.\,\eqref{eqn:lagrange-error} (which requires assuming that
$f\circ y \in C^M$) in order to yield
\begin{equation}
\label{eqn:error-III}
\left| {\bf III} \right| = \left| H_n- \int_{t_{n-1}}^{t_{n}} f(y(t))dt \right| = |H_n - I_n | \leq \frac{F}{M!} h^{M+1},
\end{equation}
where $F$ is any number that satisfies $\norm{ \frac{d^M}{dt^M} \left( f\circ y \right) }_\infty \leq F$.

Finally, we address the second, and most interesting term on the right hand side of Eqn.\,\eqref{eqn:trap-error}.
The key observation comes from recognizing this term as the difference between a low-order (local) quadrature,
$T_n$, and a high-order (global) quadrature $H_n$.
Note that the exact integral of the polynomial interpolant $E^{[p]}$
over the subinterval $[t_{n-1}, t_n]$ is
\begin{equation}
	\int_{t_{n-1}}^{t_n} E^{[p]}(t)\, dt = H^{[p]}_n - H_n,
\end{equation}
and that
\begin{equation}
	T^{[p]}_n - T_n = \frac{h_n}{2} \left( E^{ [p] }( t_n ) + E^{ [p] }( t_{n-1} ) \right)
\end{equation}
is a low-order approximation to this integral.  By textbook results, we have
\begin{equation}
\label{eqn:lte-trap}
	{\bf II} = \left(H_n^{[p]} - H_n - T_n^{[p]}+T_n \right) = -\frac{h_n^3}{12} \frac{d^2 E^{[p]} }{ dt^2 }(\xi_n),
\end{equation}
where $\xi_n$ is some number between $t_{n-1}$ and $t_n$.  Together, this implies
\begin{equation}
\label{eqn:error-II}
	\left| {\bf II} \right| \leq \frac{h_n^3}{12} \norm{ \frac{d^2 E^{[p]}}{dt^2} ( \cdot ) }.
\end{equation}

All together, inserting equations \eqref{eqn:error-I}, \eqref{eqn:error-II}, and \eqref{eqn:error-III} into \eqref{eqn:trap-error}, we have
\begin{equation}
\begin{aligned}
	\left| e_{n}^{[p+1]} \right| &\leq \left| e_{n-1}^{[p+1]} \right| + \left| {\bf I} \right| + \left| {\bf II} \right| + \left| {\bf III} \right| \\
&\leq
	\left( 1 + \frac{L h_n}{2} \right) e^{[p+1]}_{n-1} + \frac{L h_n }{ 2 } e^{[p+1]}_{n} +
	\frac{h_n^3}{12} \norm{ \frac{d^2 E^{[p]}}{dt^2} ( \cdot ) } + 
	\frac{F}{M!} h^{M+1}.
%
\end{aligned}
\end{equation}
After replacing each $h_n \leq h$, rearranging, and assuming that $\frac{h L}{2} < 1$, 
we have
\begin{equation}
\label{eqn:trap-proof-a}
\left| e_n^{[p+1]} \right| \leq 
	\underbrace{ \frac{ 1 + h L / 2 }{ 1 - hL/2 } }_{\leq e^{2hL}} e_{n-1}^{[p+1]} + \underbrace{ \frac{1}{ 1 - hL/2} }_{\leq 2} \left( \frac{h^3}{12} \norm{ \frac{d^2 E^{[p]}}{dt^2} ( \cdot ) } + \frac{F}{M!} h^{M+1} \right).
\end{equation}

We now verify the two underscored inequalities involving the $1 \pm h L / 2$ terms asserted in \eqref{eqn:trap-proof-a}.  Identical to \eqref{eqn:geo-estimate}, we have
\begin{equation*}
	\frac{1}{1-hL/2} \leq 1 + 2 \left( h L / 2 \right) = 1 + h L,
\end{equation*}
after expanding the rational expression in terms of a geometric series, and assuming that $hL/ 2 < 1$ in order to retain convergence.
Because $1 + hL/2 \leq 1 + hL$, we have
\[
	\frac{ 1 + h L / 2 }{ 1 - hL/2 } \leq (1 + h L )^2 \leq e^{2 h L},
\]
which verifies the first of the two underscored inequalities.  For the second one, we need only assume that $h L < 1$, which yields $ 1 / (1 - hL/2) < 2 $.

All together, we have
\begin{equation}
\left| e_n^{[p+1]} \right| \leq e^{2hL} e_{n-1}^{[p+1]} + 2 \left( 
	\frac{h^3}{12} \norm{ \frac{d^2 E^{[p]}}{dt^2} ( \cdot ) } + 
	\frac{F}{M!} h^{M+1} \right),
\end{equation}
which yields the desired result after appealing to Corollary \ref{cor:geo-estimate} and using $n \leq N$.

\end{proof}

\begin{rmk}
The same decomposition for the error can be used for SDC coupled with Forward (or Backward) Euler.  
That is, identical to the decomposition found in \eqref{eqn:trap-error}, we can decompose the error as
\begin{subequations}
\label{eqn:euler-error-decomp}
\begin{align}
	e_{n}^{[p+1]} &= e_{n-1}^{[p+1]} + L_n^{[p+1]} - L_n^{[p]} + H_n^{[p]} - \int_{t_{n-1}}^{t_{n}} f(y(t))dt \\ 
	&= e_{n-1}^{[p+1]} + T_n^{[p+1]}-L_n +L_n-L_n^{[p]} + H_n^{[p]}-H_n + H_n- \int_{t_{n-1}}^{t_{n}} f(y(t))dt \\ 
	&= e_{n-1}^{[p+1]} + \underbrace{\left(L_n^{[p+1]}-L_n\right)}_{\bf I} +
		\underbrace{ \left(H_n^{[p]}-L_n^{[p]}+L_n-H_n\right)}_{\bf II}  + 
		\underbrace{ I_n }_{\bf III}, 
\end{align}
\end{subequations}
where $L_n$ and $L_n^{[p]}$ denote a ``low-order" integral of the right hand side, but this time we use
\begin{equation}
	L_n := h_{n} f_{n-1}, \quad L^{[p]}_n := h_n f^{[p]}_{n-1},
\end{equation}
for Forward Euler, or instead
\begin{equation}
	L_n := h_{n} f_{n}, \quad L^{[p]}_n := h_n f^{[p]}_{n}
\end{equation}
for Backward Euler.
The third term ${\bf III}$ is again $\BigOh( h^{M+1} )$, and the first term ${\bf I}$ can be bounded by a constant times $\left| e_n^{[p]} \right| + \left| e_{n-1}^{[p+1]} \right|$.
The lack of additional order pickup can be found by observing that the second source of error instead satisfies
\begin{equation}
\label{eqn:lte-fe}
	{\bf II} = \left(H_n^{[p]} - H_n - L_n^{[p]}+L_n \right) = -\frac{h_n^2}{2} \frac{d E^{[p]} }{ dt }(\xi_n),
\end{equation}
where $\xi_n$ is some number between $t_{n-1}$ and $t_n$.  In the following examples, we compare this term to that found from the Trapezoidal rule.
\end{rmk}

\subsection{Examples}

To illustrate the results of the Theorem presented in this section, we consider the linear test case
\begin{equation}
y'(t) = y(t), \quad t > 0, \quad y(0) =1,
\end{equation}
and examine the errors produced by the SDC method when coupled with a higher order base solver.  The order pickup for the SDC method can be found by examining the size of second source of error,
defined in Equations \eqref{eqn:lte-trap} and \eqref{eqn:lte-fe},
given by $h_n^3 / 12 \left( E^{[p]} \right)''(\xi_n)$ for the trapezoidal rule, and $h_n^2 / 2 \left( E^{[p]} \right)'(\xi_n)$ for the Forward (or Backward) Euler method.  
Note that the form and size of this error is identical for either the Forward or Backward Euler base solver.

In the following examples, we work out the size of this term for a few different case studies.  Taylor expansions are found by making use of the Maple software package.

\subsubsection{$M=3$ uniformly spaced points}

We first consider the results of using a total of $M=3$ equispaced quadrature points and look at the local {\bf Error} over the subinterval $[t_{n-1},t_n]$ produced by the second term
${\bf II}$ defined in Equations 
\eqref{eqn:lte-trap} and \eqref{eqn:lte-fe}.  
In the second set of columns in Table \ref{tab:correction-errors-a} look at the size of this term by writing out Taylor expansions for the first $p=0$ SDC iteration of a solver constructed
by taking a Forward Euler provisional solution for the first time step.  Note that in this case, each point satisfies $e^{[0]}_n = \BigOh( h^2 )$ for each $n$, because the local truncation error (LTE) for Euler's method
is second order accurate.  Therefore, the jump from $\BigOh( h^2 )$ to $\BigOh( h^4 )$ indicates that the Trapezoidal correction picks up an additional two orders of accuracy, whereas the jump from $\BigOh( h^2 )$ to $\BigOh( h^3 )$ picks up a single
additional order of accuracy for the Euler base solver, which is consistent with the theory.

\begin{table}[!ht]
\caption{Correction errors with $M=3$ uniformly spaced points.
\label{tab:correction-errors-a}}
{\footnotesize 
\begin{tabular}{|c|c|c|c|c|}
\hline
{\bf Interval} & \multicolumn{2}{|c|}{{\bf Trapezoidal Rule}} & \multicolumn{2}{c|}{{\bf Forward Euler}} \\ \hline \hline
	& {\bf Error}
	& $p=0$ 
	& {\bf Error }
	& $p=0$ \\
\hline \hline
$\left[0, \frac{h}{2}\right]$ & 
$\frac{h}{24} \left( e_1^{[p]} - 2e_2^{[p]} + e_3^{[p]} \right)$ & $\BigOh( h^4 )$ & 
$\frac{h}{24} \left( 7 e^{[p]}_1 - 8 e_2^{[p]} + e_3^{[p]} \right)$ & $\BigOh( h^3 )$ \\ \hline
$[\frac h 2, h]$ &
$\frac{h}{24} \left( e_1^{[p]} - 2 e_2^{[p]} +   e_3^{[p]} \right)$ & $\BigOh( h^4 )$ &
$\frac{h}{24} \left( e_1^{[p]} + 4 e_2^{[p]} - 5 e_3^{[p]} \right)$ & $\BigOh( h^3 )$ \\
\hline
\end{tabular}}
\end{table}

\subsubsection{$M=4$ uniformly spaced points}

We next consider the same problem with a total of $M=4$ equispaced quadrature points.  Again, we look at the errors for the Trapezoidal method compared to the Forward Euler method after first 
constructing a provisional solution with the Forward Euler method.  We again observe that the Trapezoidal rule improves the order of accuracy of the provisional solution by two factors, whereas Forward Euler (or likewise Backward Euler)
only improves the order by one.  
Results for these quadrature points are presented in Table \ref{tab:correction-errors-b}.


\begin{table}[!ht]
\caption{Correction errors with $M=4$ uniformly spaced points.
\label{tab:correction-errors-b}}
{\footnotesize 
\begin{tabular}{|c|c|c|c|c|}
\hline
{\bf Interval} & \multicolumn{2}{|c|}{{\bf Trapezoidal Rule}} & \multicolumn{2}{c|}{{\bf Forward Euler}} \\ \hline \hline
	& {\bf Error}
	& $p=0$ 
	& {\bf Error }
	& $p=0$ \\
\hline \hline
$\left[0, \frac{h}{3}\right]$ & 
$\frac{h}{72} \left(  3 e_1^{[p]} - 7 e_2^{[p]} + 5 e_3^{[p]} - e_4^{[p]} \right)$ 
& $\BigOh( h^4 )$ & 
$\frac{h}{72} \left( 15 e^{[p]}_1 - 19 e_2^{[p]} + 5 e_3^{[p]} - e_4^{[p]} \right)$ 
& $\BigOh( h^3 )$ \\ \hline
$[\frac h 3, \frac {2h}{3} ]$ &
$\frac{h}{72} \left( e_1^{[p]} - e_2^{[p]} - e_3^{[p]} + e_4^{[p]} \right)$ 
& $\BigOh( h^4 )$ &
$\frac{h}{72} \left( e_1^{[p]} + 11 e_2^{[p]} - 13 e_3^{[p]} + e_4^{[p]} \right)$ 
& $\BigOh( h^3 )$ \\
\hline
$[\frac{2 h}{3}, h ]$ &
$\frac{h}{72} \left( e_1^{[p]} - 5 e_2^{[p]} + 7 e_3^{[p]} - 3 e_4^{[p]} \right)$ 
& $\BigOh( h^4 )$ &
$\frac{h}{72} \left( e_1^{[p]} - 5 e_2^{[p]} - 5 e_3^{[p]} + 9 e_4^{[p]} \right)$ 
& $\BigOh( h^3 )$ \\
\hline
\end{tabular}}
\end{table}

\subsubsection{$M=4$ non-equispaced spaced points}

Finally, we consider a case with a total of $M=4$ non-equispaced points.  As an illustrative example, we consider the quadrature points $\xi_1 = 0, \xi_2 = 1/3, \xi_3 = 1/2$, and $\xi_4 = 1$.  
We find that the trapezoidal error only increases the order of the solver by one degree, which is consistent with the findings in \cite{ChriOngQiu09}, where the authors show that 
when the second-order Runge-Kutta method is used as a corrector, then the solver does not always pick up two orders of accuracy with each correction loop.
Results for this problem are presented in Table \ref{tab:correction-errors-c}.


\begin{table}[!ht]
\caption{Correction errors with $M=4$ non-equispaced points.  In this case, we find that both methods only pick up a single additional order of accuracy.
\label{tab:correction-errors-c}}
{\footnotesize 
\begin{tabular}{|c|c|c|c|c|}
\hline
{\bf Interval} & \multicolumn{2}{|c|}{{\bf Trapezoidal Rule}} & \multicolumn{2}{c|}{{\bf Forward Euler}} \\ \hline \hline
	& {\bf Error}
	& $p=0$ 
	& {\bf Error }
	& $p=0$ \\
\hline \hline
$\left[0, \frac{h}{3}\right]$ & 
$\frac{h}{162} \left(  8 e_1^{[p]} - 27 e_2^{[p]} + 20 e_3^{[p]} - e_4^{[p]} \right)$ 
& $\BigOh( h^3 )$ & 
$\frac{h}{162} \left( 35 e^{[p]}_1 - 54 e_2^{[p]} + 20 e_3^{[p]} - e_4^{[p]} \right)$ 
& $\BigOh( h^3 )$ \\ \hline
$[\frac h 3, \frac {h}{2} ]$ &
$\frac{h}{5184} \left( 14 e_1^{[p]} -  27 e_2^{[p]} +   8 e_3^{[p]} + 5 e_4^{[p]} \right)$ 
& $\BigOh( h^3 )$ & 
$\frac{h}{5184} \left( 14 e_1^{[p]} + 405 e_2^{[p]} - 424 e_3^{[p]} + 5 e_4^{[p]} \right)$ 
& $\BigOh( h^3 )$ \\
\hline
$[\frac{h}{2}, h ]$ &
$\frac{h}{192} \left( 10 e_1^{[p]} - 81 e_2^{[p]} + 88 e_3^{[p]} - 17 e_4^{[p]} \right)$ 
& $\BigOh( h^3 )$ & 
$\frac{h}{192} \left( 10 e_1^{[p]} - 81 e_2^{[p]} + 40 e_3^{[p]} + 31 e_4^{[p]} \right)$ 
& $\BigOh( h^3 )$ \\
\hline
\end{tabular}
}
\end{table}

\section{Numerical results}
\label{section:num-results}

The primary contribution of this work is to construct rigorous error estimates for classical SDC methods,
and therefore we only include a couple of numerical results.
An abundance of SDC examples applied to ordinary and partial differential equations can be found in the literature.
One of our key goals here is
to promulgate the fact that the primary source of high-order accuracy inherent in all SDC methods comes from its underlying Picard integral formulation, and not necessarily the ``base solver," and therefore we focus our results on nearby variations of classical SDC methods
and demonstrate how classical SDC methods can be extended to produce related high-order solvers.

First, we introduce a comparison 
of errors (and stability regions) for explicit SDC vs.\,Picard iteration, second we explore modifications of the constant in front of an implicit SDC method, and 
finally we compare semi-implicit SDC and modified semi-implicit SDC solvers.
For the sake of brevity the proposed modifications to SDC methods are not formally analyzed but straightforward extensions
of the theorems presented in this work can be constructed to present formal error bounds for these methods.
The numerical evidence presented here supports this claim.

\subsection{A comparison of explicit SDC and Picard iteration}

In this numerical example, we compare the errors and stability regions by applying the Picard
iterative method defined in \eqref{eqn:picard-explicit} to that of the explicit SDC
defined in \eqref{eqn:EXP_p}.  In order to present an equal comparison of these two solvers,
we consider identical initial guesses, or provisional solutions
${\boldsymbol \eta}^{[0]}$ based upon forward Euler time stepping and we work with uniform quadrature points for all of our test cases.

\subsubsection{Errors for a linear test case}

In Figure \ref{fig:lambda-conv}, we compare errors for the linear equation
\begin{equation}
\label{eqn:linear-lambda}
y' = \lambda y, \quad y(0) = 1
\end{equation}
at a final time of $T = 10$ with $\lambda = -2$ and $\lambda = -5$.
Other orders and values of $\lambda$ show similar results
where we observe slightly smaller error constants when using the Picard iterative method compared to the
equivalent SDC method.
This is consistent with the findings of Corollary \ref{cor:error-exp}, because the first error estimate 
\begin{equation*}
|e^{[p+1]}_n| \leq |e^{[p+1]}_{n-1}| + h_n
    \left| f(\eta_{n-1}^{[p+1]})-f(\eta_{n-1}^{[p]})\right| + |I_n|
\end{equation*}
could be tightened up to read
\begin{equation*}
|e^{[p+1]}_n|\leq |e^{[p+1]}_{n-1}| + |I_n|
\end{equation*}
which produces a smaller (provable) overall error for the Picard method when compared to the SDC method.

\begin{figure}
\begin{center}
\begin{tikzpicture}[scale=0.7]
\begin{loglogaxis}[title={$\lambda=-2$},
	ymin=1e-15,ymax=1e1,xlabel={$h$},ylabel={Error $e_h$},
	grid=major,
	legend style={at={(0.98,0.02)},
	anchor=south east,font=\footnotesize,
	rounded corners=2pt}]
\addplot [mark=ball, color=blue, style=thick] table[x index = 0, y index = 1] {convergence_explicit_sdc_integrator_lam_2_order_3.dat};
\addplot [mark=square, color=red, style=thick] table[x index = 0, y index = 1] {convergence_explicit_sdc_integrator_lam_2_order_4.dat};
\addplot [mark=x, color=brown, style=thick] table[x index = 0, y index = 1] {convergence_explicit_sdc_integrator_lam_2_order_5.dat};
\addplot [mark=star, color=black, style=thick] table[x index = 0, y index = 1] {convergence_explicit_sdc_integrator_lam_2_order_6.dat};
\addplot+[dotted, mark=ball] table[x index = 0, y index = 1] {convergence_explicit_picard_iteration_lam_2_order_3.dat};
\addplot+[dotted, mark=square] table[x index = 0, y index = 1] {convergence_explicit_picard_iteration_lam_2_order_4.dat};
\addplot+[dotted, mark=x] table[x index = 0, y index = 1] {convergence_explicit_picard_iteration_lam_2_order_5.dat};
\addplot+[dotted, mark=star] table[x index = 0, y index = 1] {convergence_explicit_picard_iteration_lam_2_order_6.dat};
\legend{$M=3$, $M=4$, $M=5$, $M=6$ }
\end{loglogaxis}
\end{tikzpicture}
\begin{tikzpicture}[scale=0.7]
\begin{loglogaxis}[title={$\lambda=-5$},ymin=1e-15,ymax=1e1,xlabel={$h$},ylabel={Error $e_h$},grid=major,legend style={at={(0.98,0.02)},anchor=south east,font=\footnotesize,rounded corners=2pt}]
\addplot [mark=ball, color=blue, style=thick] table[x index = 0, y index = 1] {convergence_explicit_sdc_integrator_lam_5_order_3.dat};
\addplot [mark=square, color=red, style=thick] table[x index = 0, y index = 1] {convergence_explicit_sdc_integrator_lam_5_order_4.dat};
\addplot [mark=x, color=brown, style=thick] table[x index = 0, y index = 1] {convergence_explicit_sdc_integrator_lam_5_order_5.dat};
\addplot [mark=star, color=black, style=thick] table[x index = 0, y index = 1] {convergence_explicit_sdc_integrator_lam_5_order_6.dat};
\addplot+[dotted, mark=ball] table[x index = 0, y index = 1] {convergence_explicit_picard_iteration_lam_5_order_3.dat};
\addplot+[dotted, mark=square] table[x index = 0, y index = 1] {convergence_explicit_picard_iteration_lam_5_order_4.dat};
\addplot+[dotted, mark=x] table[x index = 0, y index = 1] {convergence_explicit_picard_iteration_lam_5_order_5.dat};
\addplot+[dotted, mark=star] table[x index = 0, y index = 1] {convergence_explicit_picard_iteration_lam_5_order_6.dat};
\legend{$M=3$, $M=4$, $M=5$, $M=6$ }
\end{loglogaxis}
\end{tikzpicture}
\caption{Linear test case.  Here, we compare explicit SDC (solid lines) with that of Picard iteration (dashed lines) of various orders.
Each method attains the desired order of accuracy with the minimum number of corrections.
In each case, Picard iteration has slightly smaller errors when compared the equivalent explicit SDC method of the same order and same quadrature rule.
\label{fig:lambda-conv}}
\end{center}
\end{figure}
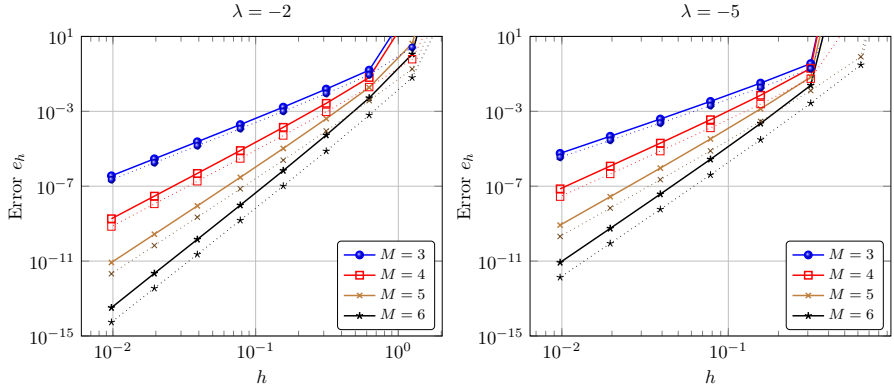

\subsubsection{A comparison of regions of absolute stability for explicit methods}

Next, we seek to compare regions of absolute stability for explicit SDC methods and their Picard iterative cousins.
Here we observe that the stability regions are slightly improved when the 
``Euler term" in the time stepping is dropped from the SDC method.  
That is to say, we find that 
the Picard iterative methods generally have larger regions of absolute stability when compared
to their SDC counterparts.

In order to demonstrate this, in Figure \ref{fig:stability-plots}, we include a comparison of plots of the regions of absolute stability, defined by
\begin{equation}
\mathbb{D} := \left\{
z \in \mathbb{C} : |\rho(z)| < 1
\right\}
\end{equation}
where $\rho(z)$ is the amplification factor for various quadrature rules for both of these methods, $z := \lambda h$,
and $\lambda$ is defined as in Eqn.\,\eqref{eqn:linear-lambda}.
There, we compare methods of orders two through ten, all based on equispaced quadrature points, forward Euler time
stepping for the provisional solution, and the minimum number of corrections required to reach the desired order
of accuracy.  (For example, the third order method uses two corrections and the fifth order method uses four corrections.) 
Similar to most explicit Runge-Kutta methods, we find that the regions of absolute stability increase as the order is increased,
but there are also more function evaluations per time step.


\begin{figure}[!ht]
\begin{center}
    \begin{tabular}{cc}
        \includegraphics[width=0.3\textwidth]{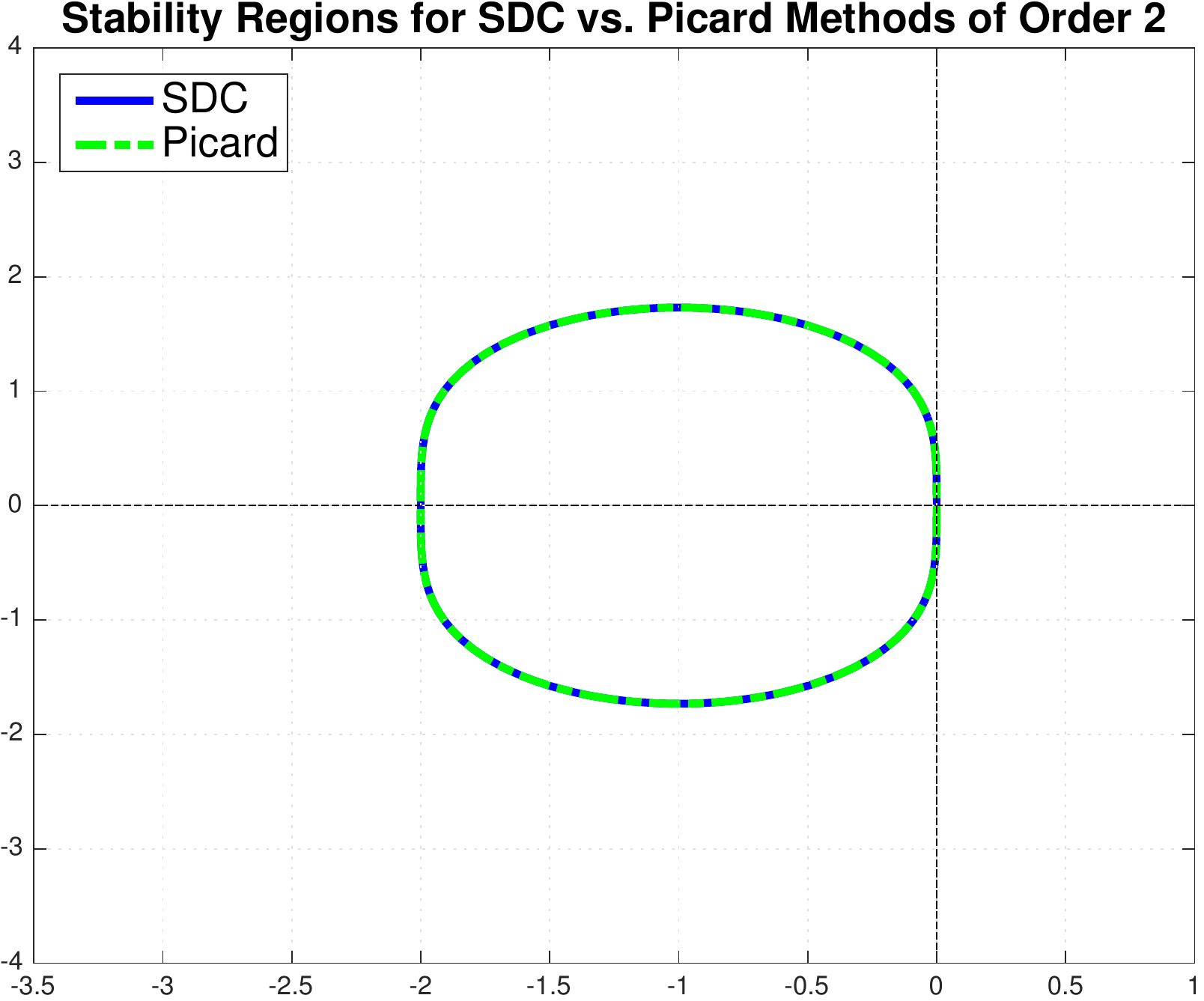}
        \includegraphics[width=0.3\textwidth]{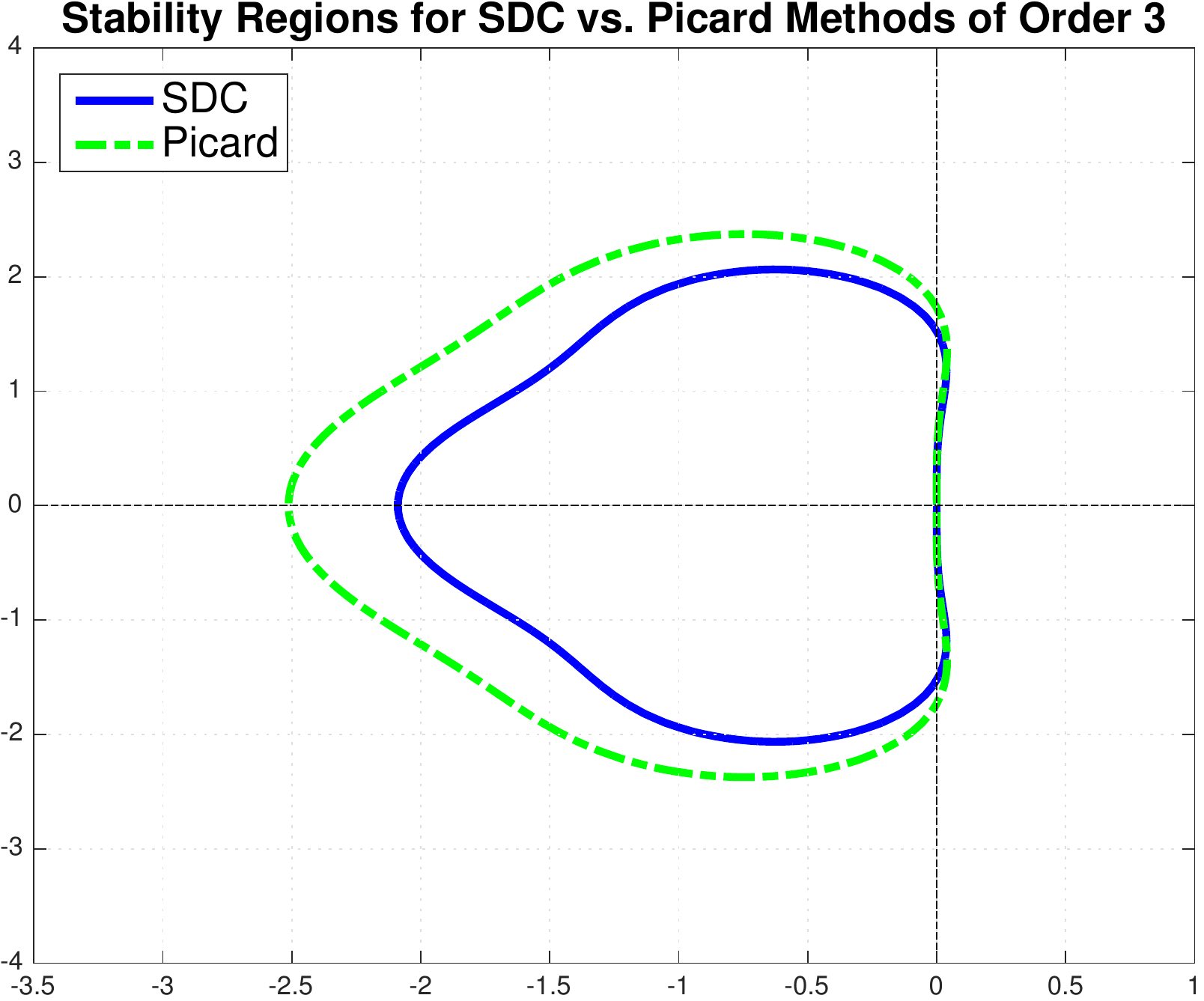}
        \includegraphics[width=0.3\textwidth]{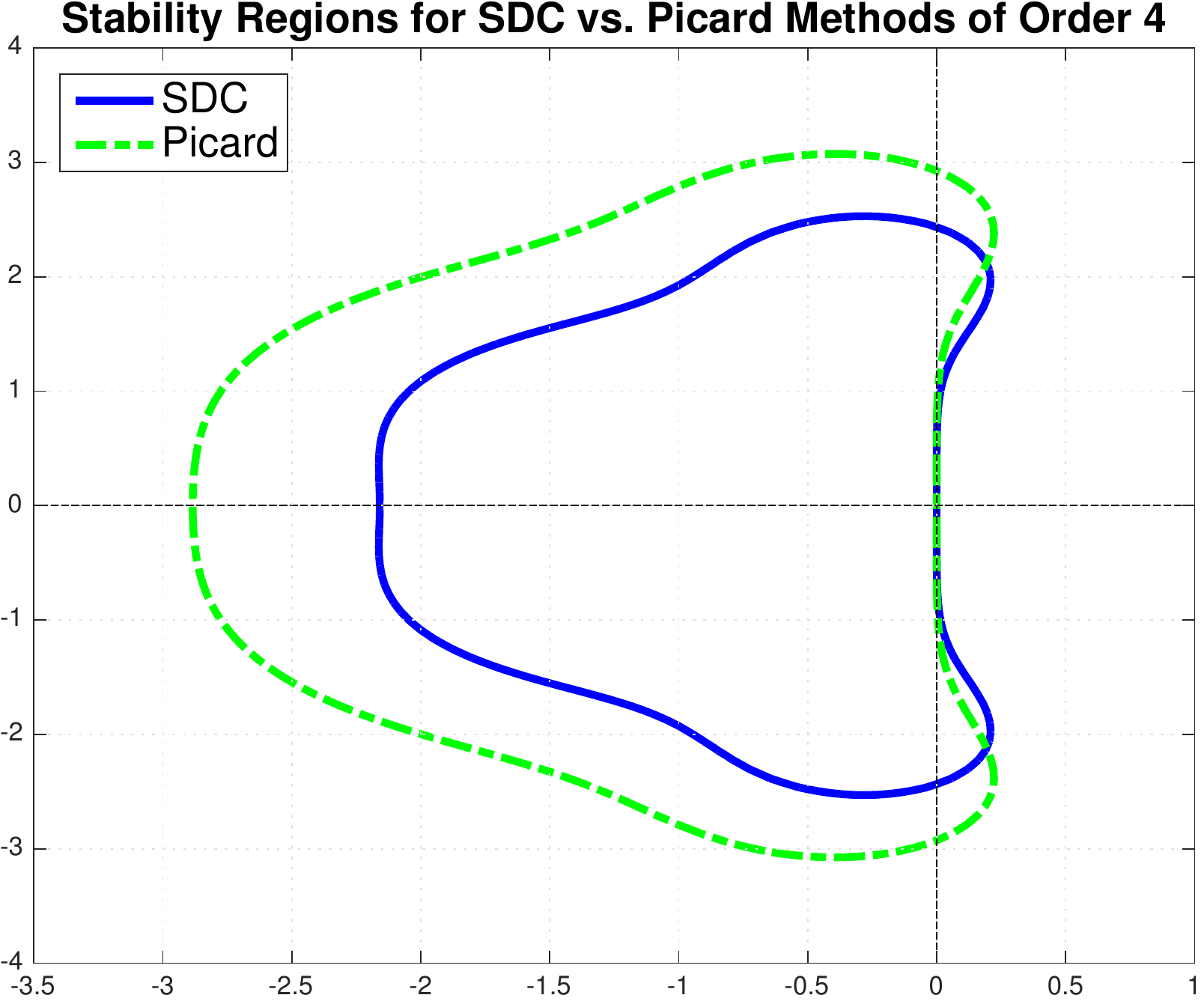} \\
        \includegraphics[width=0.3\textwidth]{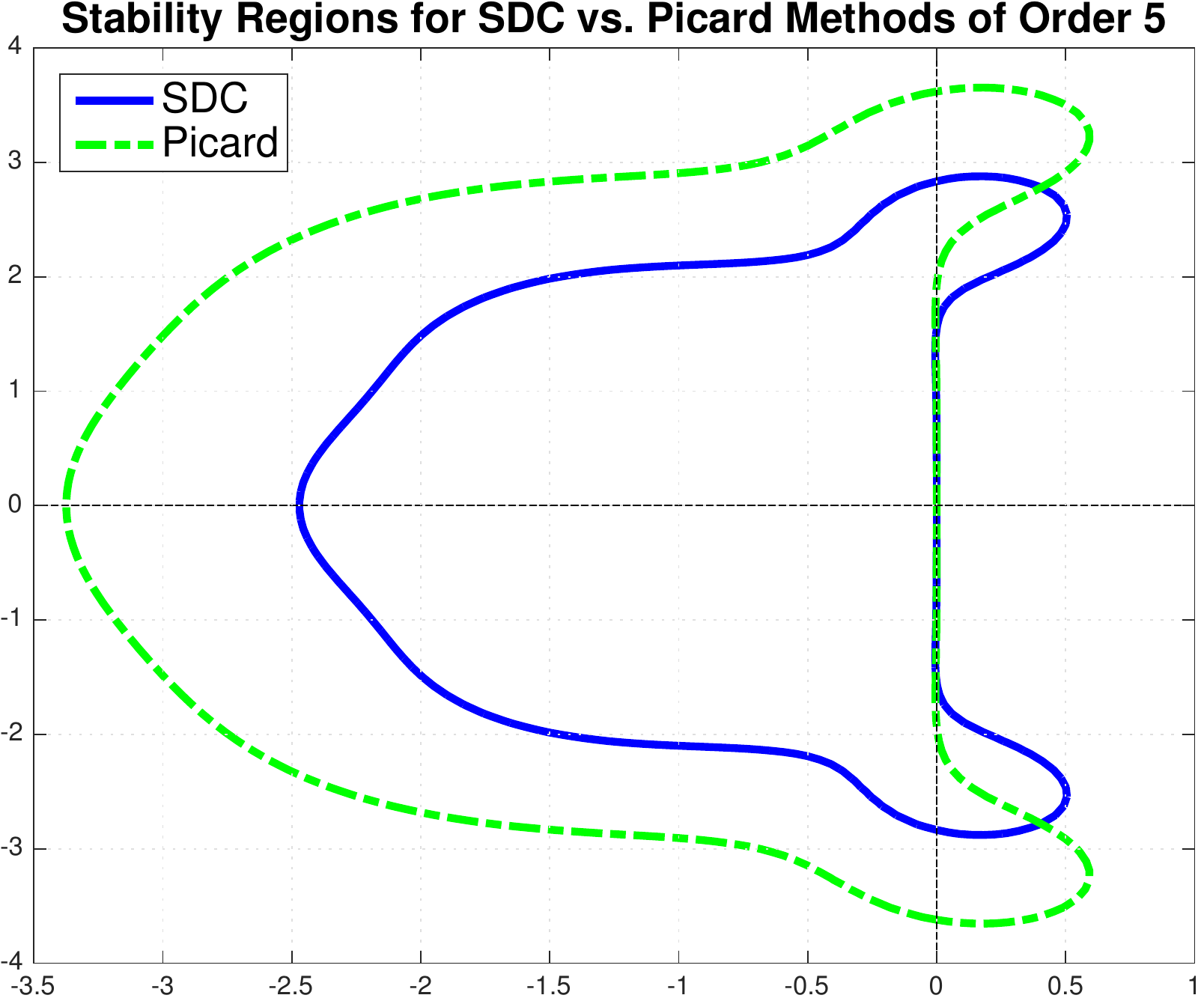}
        \includegraphics[width=0.3\textwidth]{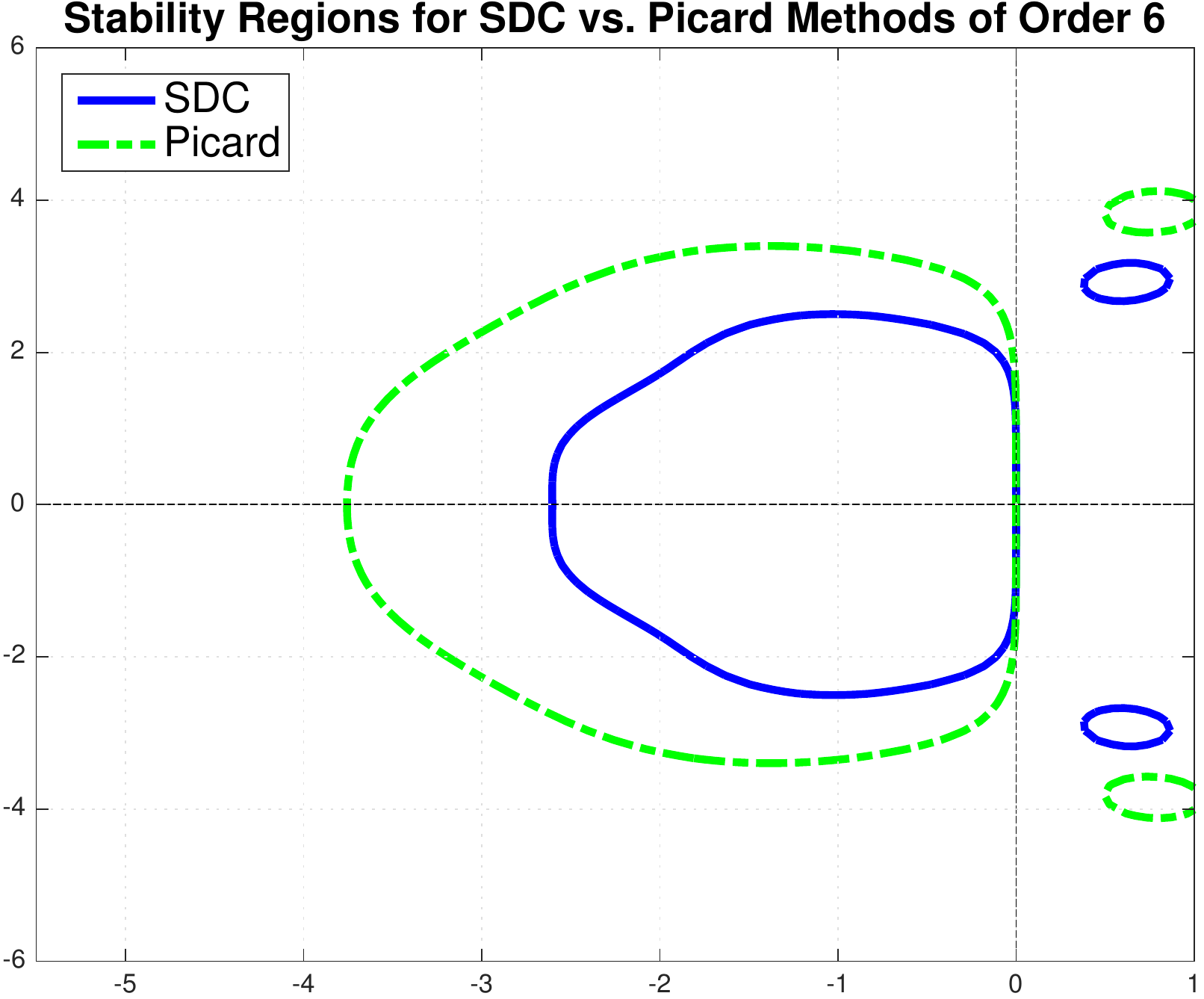}
        \includegraphics[width=0.3\textwidth]{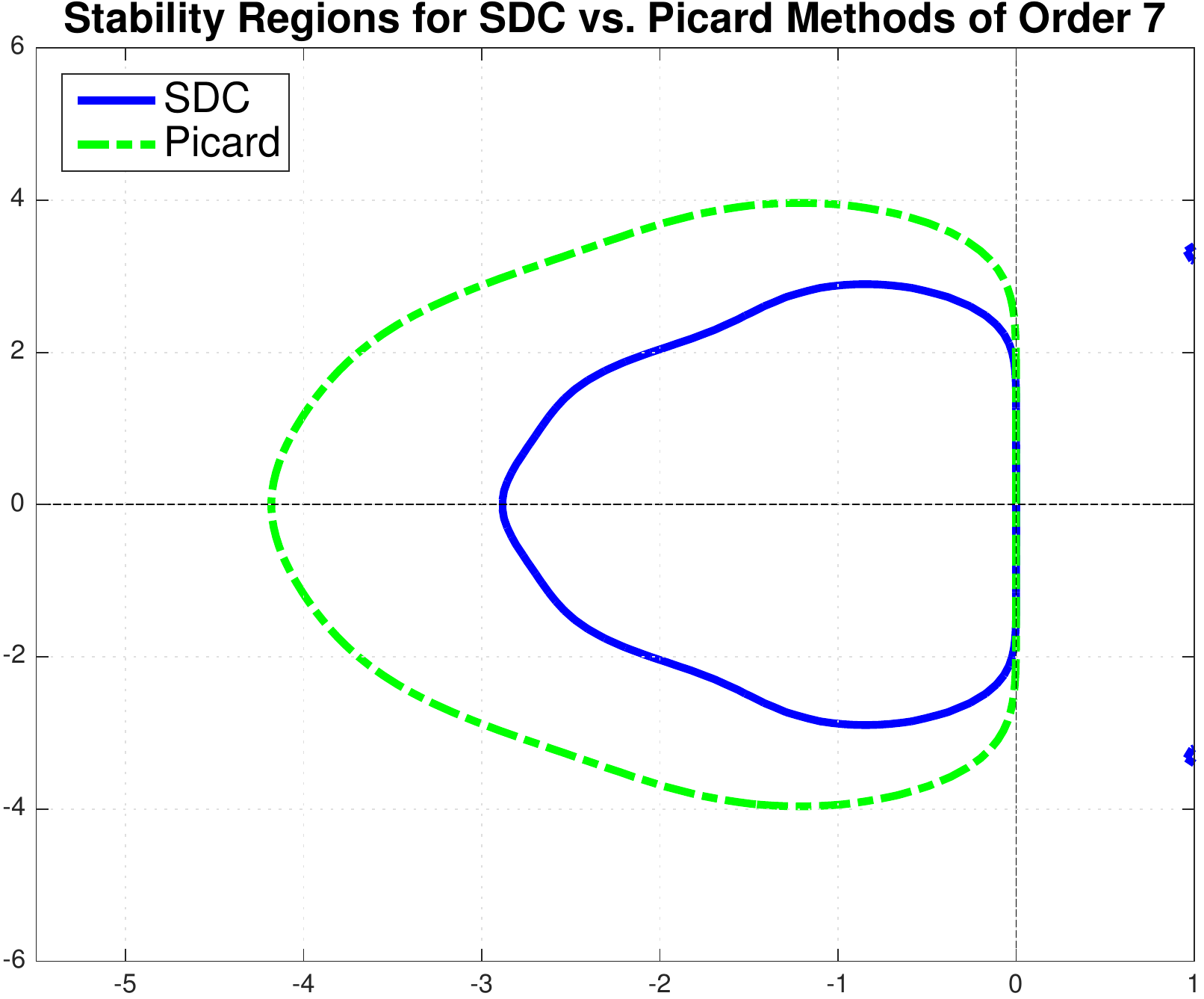} \\
        \includegraphics[width=0.3\textwidth]{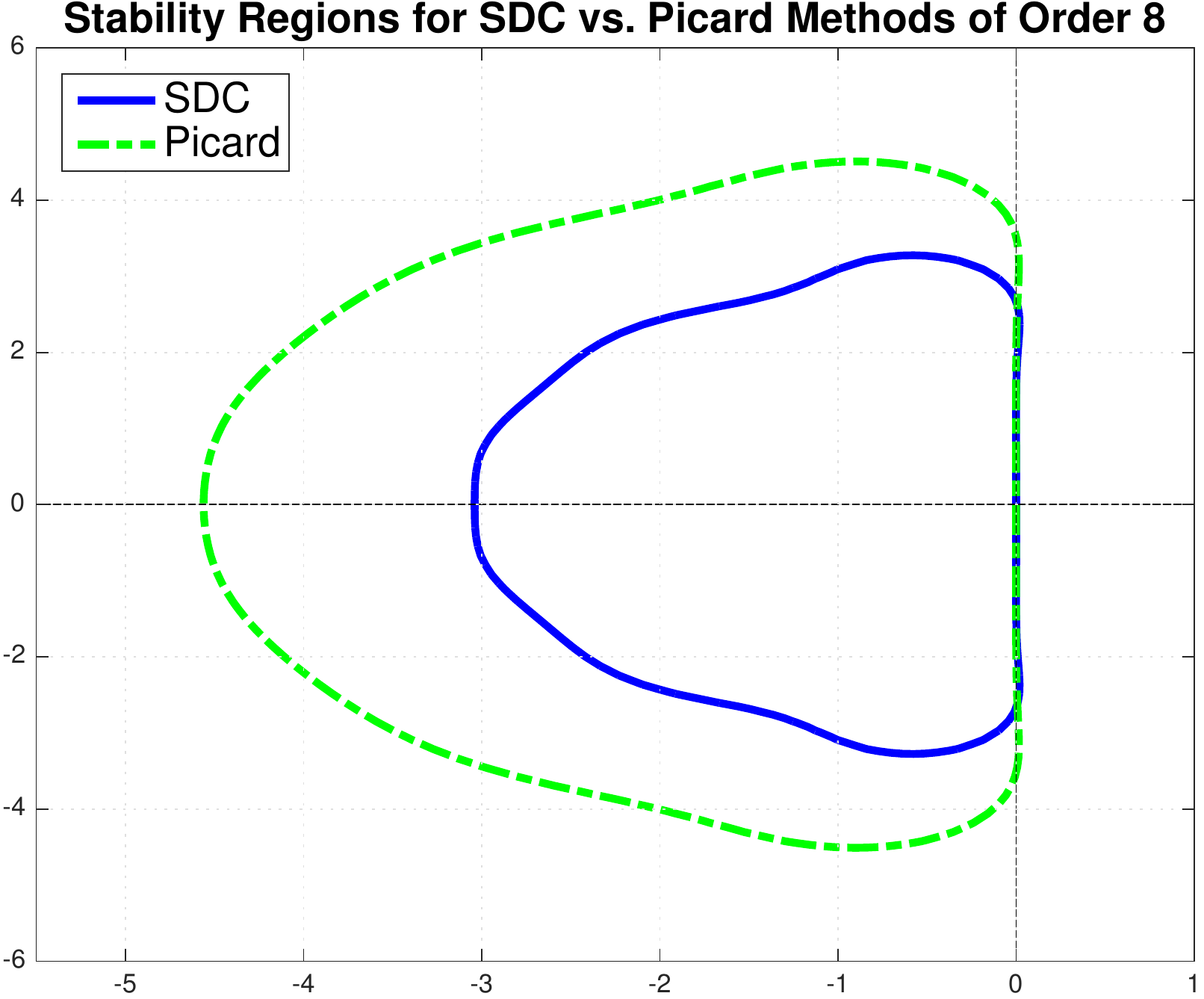}
        \includegraphics[width=0.3\textwidth]{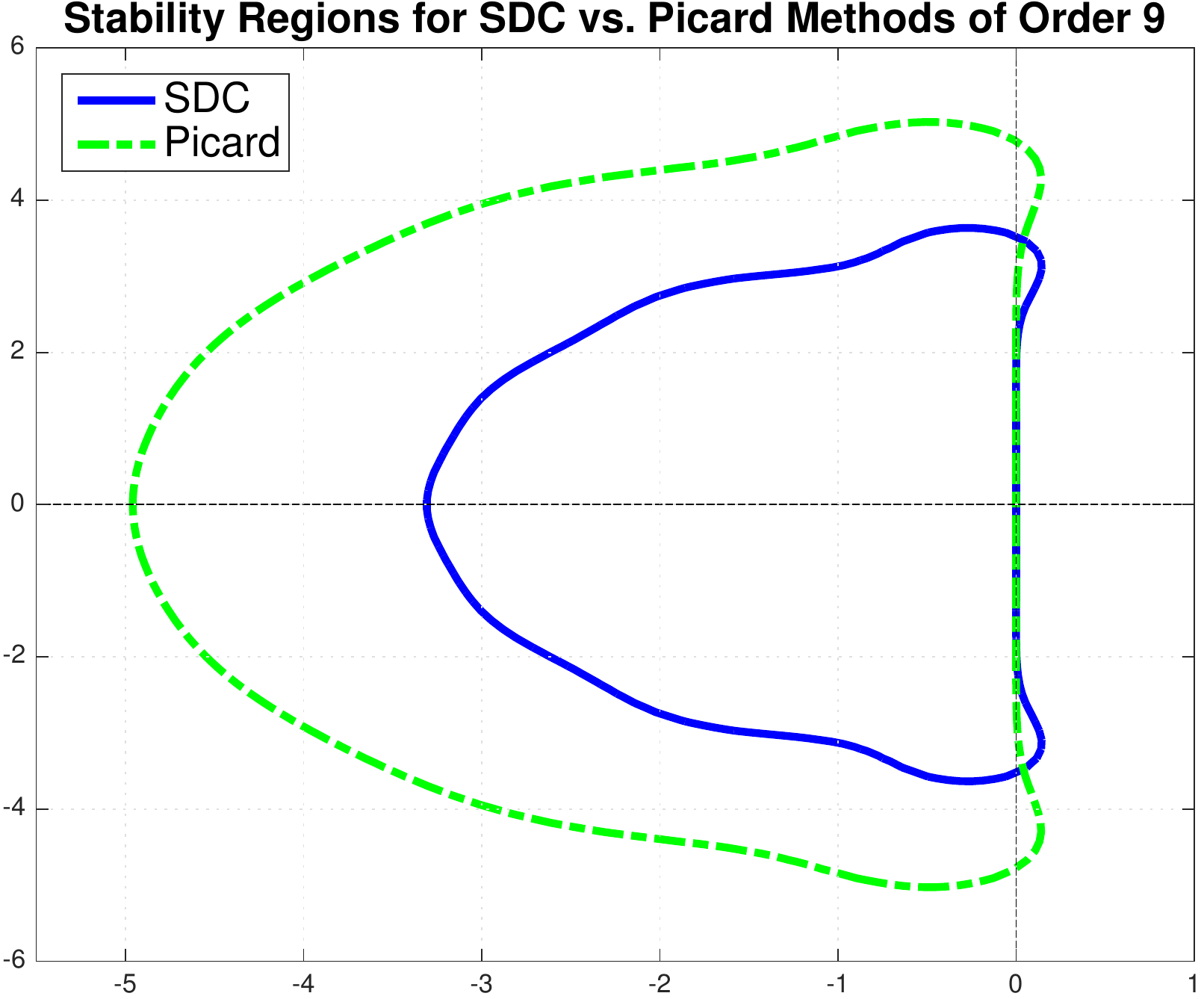}
        \includegraphics[width=0.3\textwidth]{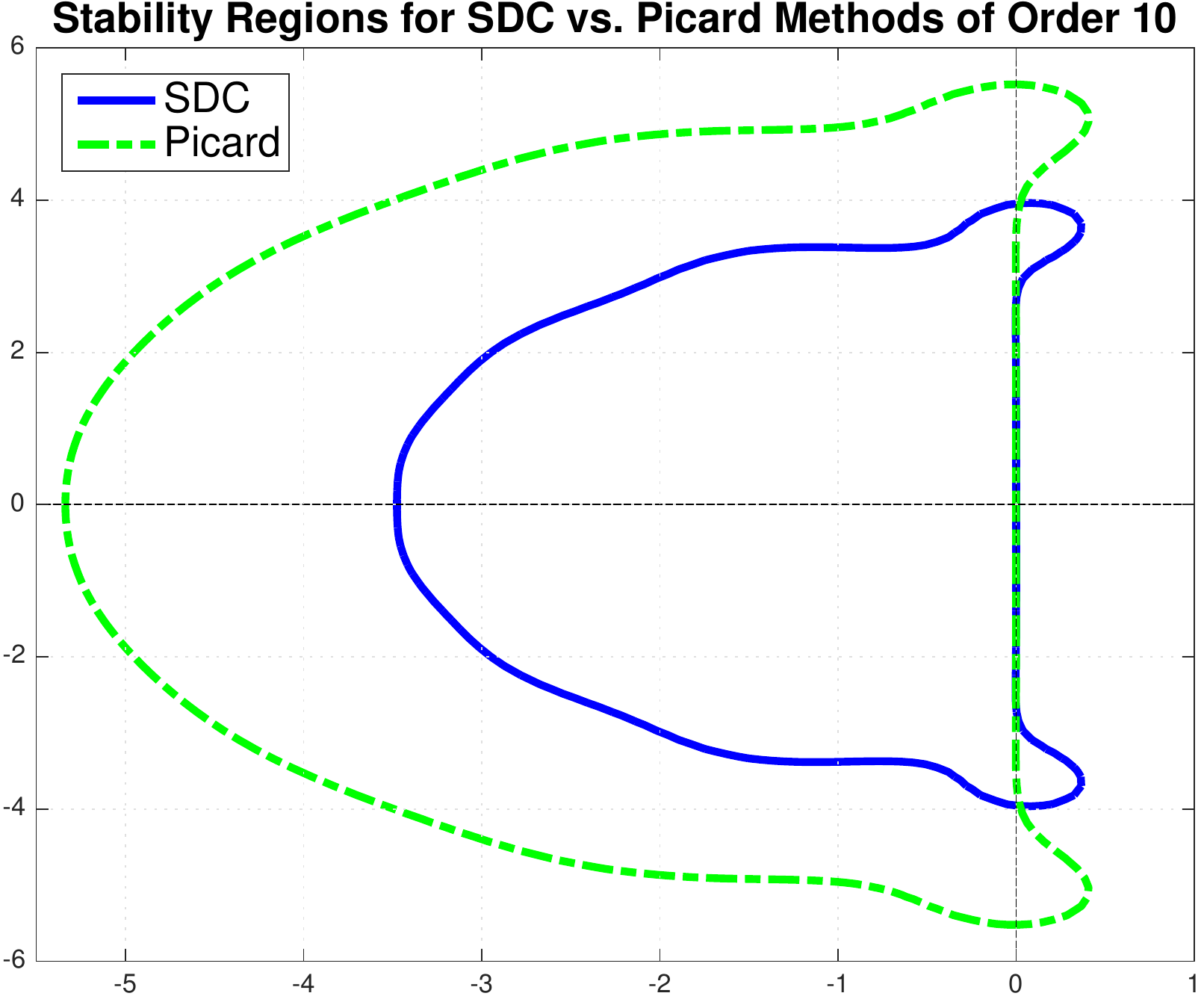} \\
    \end{tabular}
\caption{Stability regions for explicit methods.  Here, we compare SDC methods to that of Picard iterative methods where the explicit ``Euler term"
is dropped from the iterative process.  In each case save one, we observe that the Picard iterative methods have slightly larger regions of 
absolute stability.  The second-order SDC and Picard methods that use two quadrature points are identical because the forward Euler time step vanishes in the SDC method.
The scaling on the axes for the methods of orders six through ten is different than the scaling for the methods of orders two through five.
\label{fig:stability-plots}}
\end{center}
\end{figure}

\subsection{Implicit SDC methods with modified backward Euler time steps}

Here, we consider implicit SDC methods with a variable constant in front of the vanishing term:
\begin{equation}
\label{eqn:implicit-mod-be}
    \eta_n^{[p+1]} = \eta_{n-1}^{[p+1]} + \theta h
    \left[  f(\eta_{n}^{[p+1]})-f(\eta_{n}^{[p]})\right] + h \sum_{m=1}^M w_{n,m}f(\eta_m^{[p]}),
    \quad n=1,2,, \ldots N.
\end{equation}
This same scaling has already been explored in \cite{XiaXuShu07} for SDC methods, but there the authors only consider the case where $1/2 \leq \theta \leq 1$.
With $\theta = 0$, we have (explicit) Picard iteration (provided the provisional solution is modified), with $\theta = 1$, we have the classical implicit SDC method, and with 
negative values of $\theta$ we have backward Euler solves on negative time steps; none of the these changes effect the overall order of accuracy, only the
size of the error constant and the regions of absolute stability.  
Following the proof of the main Theorem in this work, we find the following result under a modified estimate for the time step size.

\begin{theorem} 
\label{thm:error-imp-modified}
The errors for a single step of the the modified implicit SDC method defined in \eqref{eqn:implicit-mod-be}
satisfy
\begin{equation}
\label{eqn:theorem-imp-mod}
        |e^{[p+1]}_n| \leq e^{2 N |\theta| h L} |e_{0}| + C_1 h \norm{ {\bf e}^{ [p] } } + C_2 h^{M+1}
\end{equation}
provided $|\theta| h < 1 / (2L)$.  
The constants
\[
C_1 = 2 N e^{2N|\theta|hL} \left( |\theta| + L \max_{n} W_n \right) \quad \text{and} \quad
C_2 = 2 N e^{2N|\theta|hL} \frac{F}{M!},
\]
again depend only on the smoothness of $f$, the exact solution $y$, the choice
of quadrature points, but this time they also depend on $\theta$.
\end{theorem}

Note that the value of $\theta = 0$ minimizes the size of these constants, but it also produces poor regions of absolute stability.
With $\theta \gg 1$ we have a method that is heavy handed
on multiple backward Euler solves, and therefore it has a very large
region of absolute stability, however these methods unfortunately introduce larger error constants.
%
%
Small values of $\theta$ decrease these error constants, but they modify the regions of absolute stability to the point
where they become finite and therefore undesirable as these are implicit methods.

\subsubsection{A verification of high-order accuracy: The nonlinear pendulum problem}

As a verification of the high-order accuracy of the solvers, we consider the equations of motion for a nonlinear pendulum:
\begin{equation}
\label{eqn:pendulum-2ndorder}
x''(t) + \sin( x(t) ) = 0
\end{equation}
with appropriate initial conditions.  If we perform the change of variables $y_1(t) = x(t)$ and $y_2(t) = x'(t)$, we end up with the following
first-order nonlinear system of equations that is equivalent to Eqn.\,\eqref{eqn:pendulum-2ndorder}:
\begin{equation}
\left( y_1, y_2 \right)' = \left( y_2, -\sin( y_1 ) \right).  
\end{equation}
We consider initial conditions defined by $(y_1(0), y_2(0)) = (0,1)$ and we integrate this problem to a final time of $T = 10$.
To compute a reference solution, we use MATLAB's builtin ode45 with a relative tolerance of $10^{-12}$ and an absolute tolerance
of $10^{-14}$.  

We present convergence results for this problem in Figure \ref{fig:imp-scalings}.  These results indicate that each method is indeed high-order independent of the value
of $\theta$.  For the sake of brevity, we only report results for methods with a total of $M=4$ equispaced quadrature points, but we also compare results for different number
of corrections.  (This underlying quadrature rule is also known as Simpson's $3/8$ rule, which has a smaller error constant
than Simpson's rule that uses $M=3$ equispaced points, but it comes at the cost of an additional function evaluation.)
We not only find that smaller values of $\theta$ produce smaller errors, which the theory supports, but we also demonstrate that more corrections 
for the methods with large $\theta$ values can help to decrease the errors.

    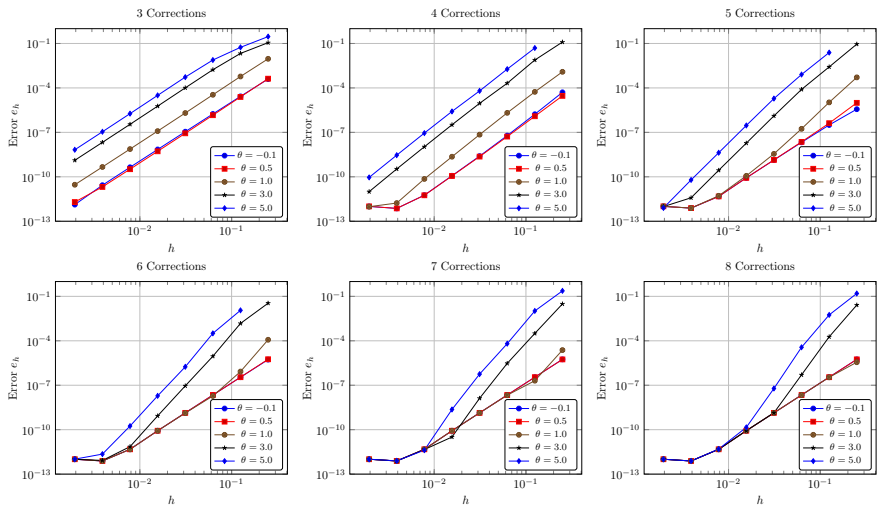
\begin{figure}[h]
    \begin{center}
    \begin{minipage}{.3\textwidth}
    \begin{tikzpicture}[scale=0.45]
    \begin{loglogaxis}[title=3 Corrections, ymin=1e-13,ymax=1e0,xlabel={$h$},ylabel={Error $e_h$},grid=major,legend style={at={(0.98,0.02)},anchor=south east,font=\footnotesize,rounded corners=2pt}]
    \addplot table[x index = 0, y index = 1] {convergence_implicit_sdc_integrator_theta_0.1.dat};
    \addplot table[x index = 0, y index = 1] {convergence_implicit_sdc_integrator_theta_0.5.dat};
    \addplot table[x index = 0, y index = 1] {convergence_implicit_sdc_integrator_theta_1.dat};
    \addplot table[x index = 0, y index = 1] {convergence_implicit_sdc_integrator_theta_3.dat};
    \addplot table[x index = 0, y index = 1] {convergence_implicit_sdc_integrator_theta_5.dat};
    \legend{$\theta=-0.1$, $\theta=0.5$, $\theta=1.0$,  $\theta=3.0$, $\theta=5.0$ }
    \end{loglogaxis}
    \end{tikzpicture}
    \end{minipage}
    \begin{minipage}{.3\textwidth}
    \begin{tikzpicture}[scale=0.45]
    \begin{loglogaxis}[title=4 Corrections, ymin=1e-13,ymax=1e0,xlabel={$h$},ylabel={Error $e_h$},grid=major,legend style={at={(0.98,0.02)},anchor=south east,font=\footnotesize,rounded corners=2pt}]
    \addplot table[x index = 0, y index = 1] {convergence_implicit_sdc_integrator_theta_0.1_ncor_4.dat};
    \addplot table[x index = 0, y index = 1] {convergence_implicit_sdc_integrator_theta_0.5_ncor_4.dat};
    \addplot table[x index = 0, y index = 1] {convergence_implicit_sdc_integrator_theta_1_ncor_4.dat};
    \addplot table[x index = 0, y index = 1] {convergence_implicit_sdc_integrator_theta_3_ncor_4.dat};
    \addplot table[x index = 0, y index = 1] {convergence_implicit_sdc_integrator_theta_5_ncor_4.dat};
    \legend{$\theta=-0.1$, $\theta=0.5$, $\theta=1.0$,  $\theta=3.0$, $\theta=5.0$ }
    \end{loglogaxis}
    \end{tikzpicture}
    \end{minipage}
    \begin{minipage}{0.3\textwidth}
    \begin{tikzpicture}[scale=0.45]
    \begin{loglogaxis}[title=5 Corrections, ymin=1e-13,ymax=1e0,xlabel={$h$},ylabel={Error $e_h$},grid=major,legend style={at={(0.98,0.02)},anchor=south east,font=\footnotesize,rounded corners=2pt}]
    \addplot table[x index = 0, y index = 1] {convergence_implicit_sdc_integrator_theta_0.1_ncor_5.dat};
    \addplot table[x index = 0, y index = 1] {convergence_implicit_sdc_integrator_theta_0.5_ncor_5.dat};
    \addplot table[x index = 0, y index = 1] {convergence_implicit_sdc_integrator_theta_1_ncor_5.dat};
    \addplot table[x index = 0, y index = 1] {convergence_implicit_sdc_integrator_theta_3_ncor_5.dat};
    \addplot table[x index = 0, y index = 1] {convergence_implicit_sdc_integrator_theta_5_ncor_5.dat};
    \legend{$\theta=-0.1$, $\theta=0.5$, $\theta=1.0$,  $\theta=3.0$, $\theta=5.0$ }
    \end{loglogaxis}
    \end{tikzpicture}
    \end{minipage} \\
    \begin{minipage}{0.3\textwidth}
    \begin{tikzpicture}[scale=0.45]
    \begin{loglogaxis}[title=6 Corrections, ymin=1e-13,ymax=1e0,xlabel={$h$},ylabel={Error $e_h$},grid=major,legend style={at={(0.98,0.02)},anchor=south east,font=\footnotesize,rounded corners=2pt}]
    \addplot table[x index = 0, y index = 1] {convergence_implicit_sdc_integrator_theta_0.1_ncor_6.dat};
    \addplot table[x index = 0, y index = 1] {convergence_implicit_sdc_integrator_theta_0.5_ncor_6.dat};
    \addplot table[x index = 0, y index = 1] {convergence_implicit_sdc_integrator_theta_1_ncor_6.dat};
    \addplot table[x index = 0, y index = 1] {convergence_implicit_sdc_integrator_theta_3_ncor_6.dat};
    \addplot table[x index = 0, y index = 1] {convergence_implicit_sdc_integrator_theta_5_ncor_6.dat};
    \legend{$\theta=-0.1$, $\theta=0.5$, $\theta=1.0$,  $\theta=3.0$, $\theta=5.0$ }
    \end{loglogaxis}
    \end{tikzpicture}
    \end{minipage}
    \begin{minipage}{0.3\textwidth}
    \begin{tikzpicture}[scale=0.45]
    \begin{loglogaxis}[title=7 Corrections, ymin=1e-13,ymax=1e0,xlabel={$h$},ylabel={Error $e_h$},grid=major,legend style={at={(0.98,0.02)},anchor=south east,font=\footnotesize,rounded corners=2pt}]
    \addplot table[x index = 0, y index = 1] {convergence_implicit_sdc_integrator_theta_0.1_ncor_7.dat};
    \addplot table[x index = 0, y index = 1] {convergence_implicit_sdc_integrator_theta_0.5_ncor_7.dat};
    \addplot table[x index = 0, y index = 1] {convergence_implicit_sdc_integrator_theta_1_ncor_7.dat};
    \addplot table[x index = 0, y index = 1] {convergence_implicit_sdc_integrator_theta_3_ncor_7.dat};
    \addplot table[x index = 0, y index = 1] {convergence_implicit_sdc_integrator_theta_5_ncor_7.dat};
    \legend{$\theta=-0.1$, $\theta=0.5$, $\theta=1.0$,  $\theta=3.0$, $\theta=5.0$ }
    \end{loglogaxis}
    \end{tikzpicture}
    \end{minipage}
    \begin{minipage}{0.3\textwidth}
    \begin{tikzpicture}[scale=0.45]
    \begin{loglogaxis}[title=8 Corrections, ymin=1e-13,ymax=1e0,xlabel={$h$},ylabel={Error $e_h$},grid=major,legend style={at={(0.98,0.02)},anchor=south east,font=\footnotesize,rounded corners=2pt}]
    \addplot table[x index = 0, y index = 1] {convergence_implicit_sdc_integrator_theta_0.1_ncor_8.dat};
    \addplot table[x index = 0, y index = 1] {convergence_implicit_sdc_integrator_theta_0.5_ncor_8.dat};
    \addplot table[x index = 0, y index = 1] {convergence_implicit_sdc_integrator_theta_1_ncor_8.dat};
    \addplot table[x index = 0, y index = 1] {convergence_implicit_sdc_integrator_theta_3_ncor_8.dat};
    \addplot table[x index = 0, y index = 1] {convergence_implicit_sdc_integrator_theta_5_ncor_8.dat};
    \legend{$\theta=-0.1$, $\theta=0.5$, $\theta=1.0$,  $\theta=3.0$, $\theta=5.0$ }
    \end{loglogaxis}
    \end{tikzpicture}
    \end{minipage}
    \end{center}
    \caption{Nonlinear pendulum problem.  Here, we compare SDC methods with different scalings on the backward Euler term
    as defined in Eqn.\,\eqref{eqn:implicit-mod-be}.  The case with $\theta = 1$ is classical implicit SDC.
    We observe the expected result that all methods have high-order accuracy and that $\theta > 1$ produces larger error constants.
    The case with $\theta = -0.1$ is not a useful method because it has a finite region of absolute stability.  
    While all methods here are fourth-order accurate after three corrections, the methods with large values of $\theta$ stand to gain the
    most through additional corrections.  This can be attributed to the large error constants found in the backward Euler term that vanish
    at the number of iterations increase (provided the iterates converge).
    \label{fig:imp-scalings}}
    \end{figure}

\subsubsection{A parameter study of regions of absolute stability for implicit methods}

Given the results of the previous section, it would be tempting to want to set $\theta = 0$ in order to reduce the total error.
What is missing from this observation is an understanding of the regions of absolute stability.  We now address this question.
We find that small values of $\theta$ produce finite regions of absolute stability, and that large values of $\theta$ increase
the regions of absolute stability (when compared to classical SDC methods) but they also increase the stiffness of each implicit solve.
With that being said, larger time steps should be able to be taken, but as is pointed out in the previous section, larger errors are introduced.
This reproduces the usual tradeoff between being able to take large time steps with large errors or being forced to take smaller time steps 
but at an increased computational cost.

In Figure \ref{fig:stability-plots-imp} we present results for a third order method with various values of $\theta$.  There we plot 
contour plots of the 
modulus of the amplification factor $| \rho(z) |$ in place of the boundary defined by $|\rho(z)| = 1$ because if we were to plot the boundary then it would not be clear what parts are stable.
In this sequence of images, we present results for various values of $\theta \in [-0.4,5]$.  Larger values of $\theta$ such as $\theta = 100$ look
very similar to that of $\theta = 5$.
Tests on methods of other orders produce similar results involving transitions between finite and infinite regions of absolute stability
as $\theta$ increases from $0$.
The tradeoff between the size and shape of the stability regions for various quadrature rules is left for future work.

Before continuing, we stop to point out that
diagonally implicit Runge-Kutta methods (on nonequispaced points) can be 
constructed from this very same framework.
The point here is that because the term involving the difference $f( \eta^{[p+1]}_n) - f( \eta^{[p]}_n)$ in 
Eqn.\,\eqref{eqn:implicit-mod-be} does not contribute to the overall order of accuracy; the scaling in front of this term
can be modified so that each implicit solve uses the exact same time step, which would result in a
singly diagonally implicit Runge-Kutta (SDIRK) method.
This could be advantageous for easing the implementation
of SDC methods in large scale code bases.  In such a case the provisional solution would have to be modified in order to retain
constant time steps for each stage in the solver.
This would mean an extra correction or a (low-order) polynomial interpolation step would be necessary to not lose the
starting accuracy found in the provisional solution, which would again modify the regions of absolute stability for solvers of various orders.

Similar modifications have recently been explored on nonequispaced points from a linear algebra perspective.
In \cite{Weiser15}, the author makes use of this vantage and optimizes their solvers by 
modifying the coefficients in the fixed point iteration matrices.
It is pointed out that there are a number
of items that could be optimized, such as the spectral radius of the solver (in order to optimize the convergence rate of the sweeps), 
the matrix norm (for the purposes of reducing the error under the assumption of a small number of sweeps),
the error at the final time, the average reduction factor in each 
sweep block (for the purposes of adaptively choosing the number of sweeps, which could include flexible or greedy sweeps), and so on.
Even though SDC methods are a subset of Runge-Kutta methods, and all of these options can be found by looking at this more general
class of methods, one key advantage SDC methods enjoy is
they do not typically sacrifice the difficult order conditions that a more generic RK method would have to address.
At the same time, SDC methods still have ample levers to tune for optimization purposes.


\begin{figure}
\begin{center}

    \begin{tabular}{cc}
        \includegraphics[width=0.3\textwidth]{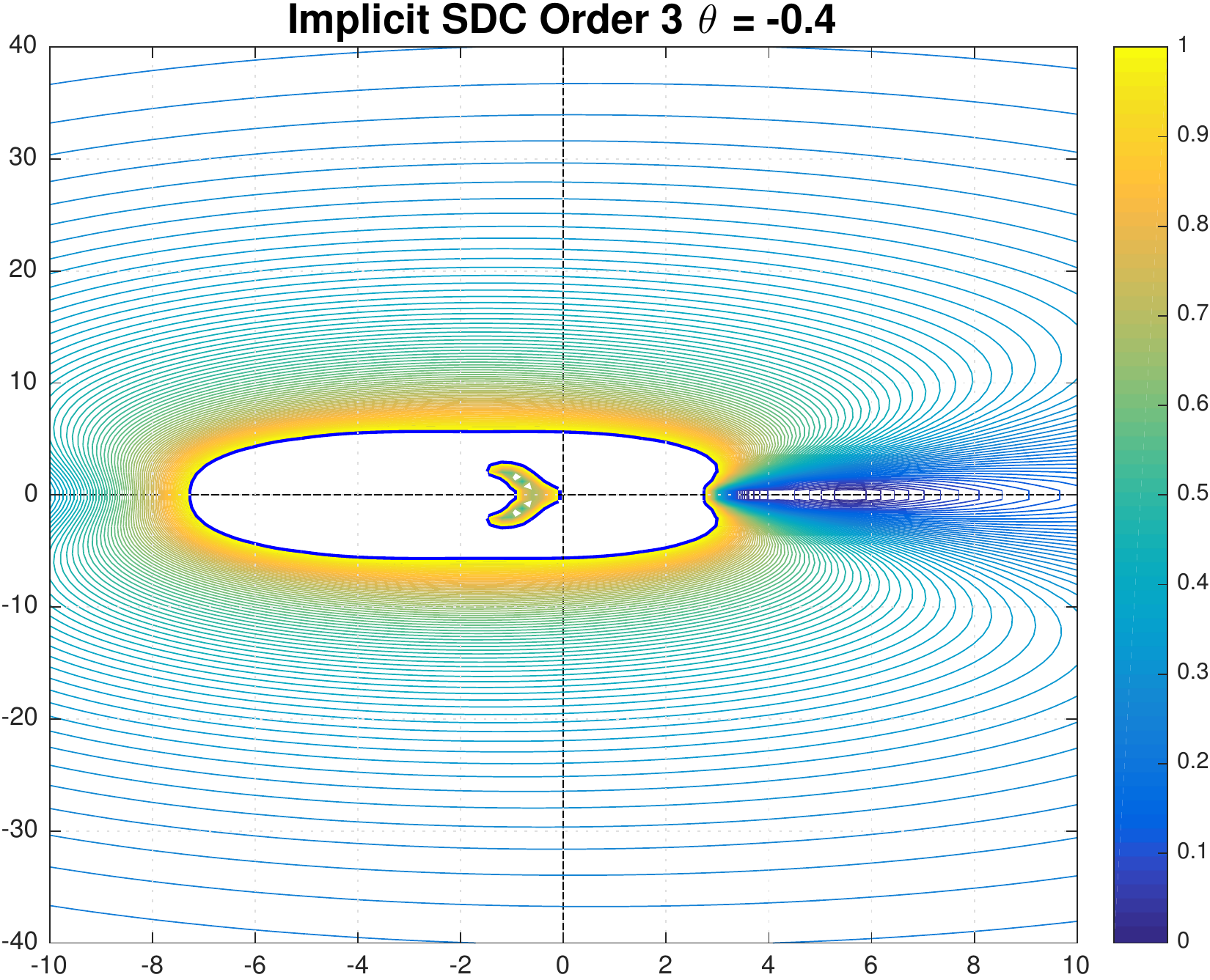}
        \includegraphics[width=0.3\textwidth]{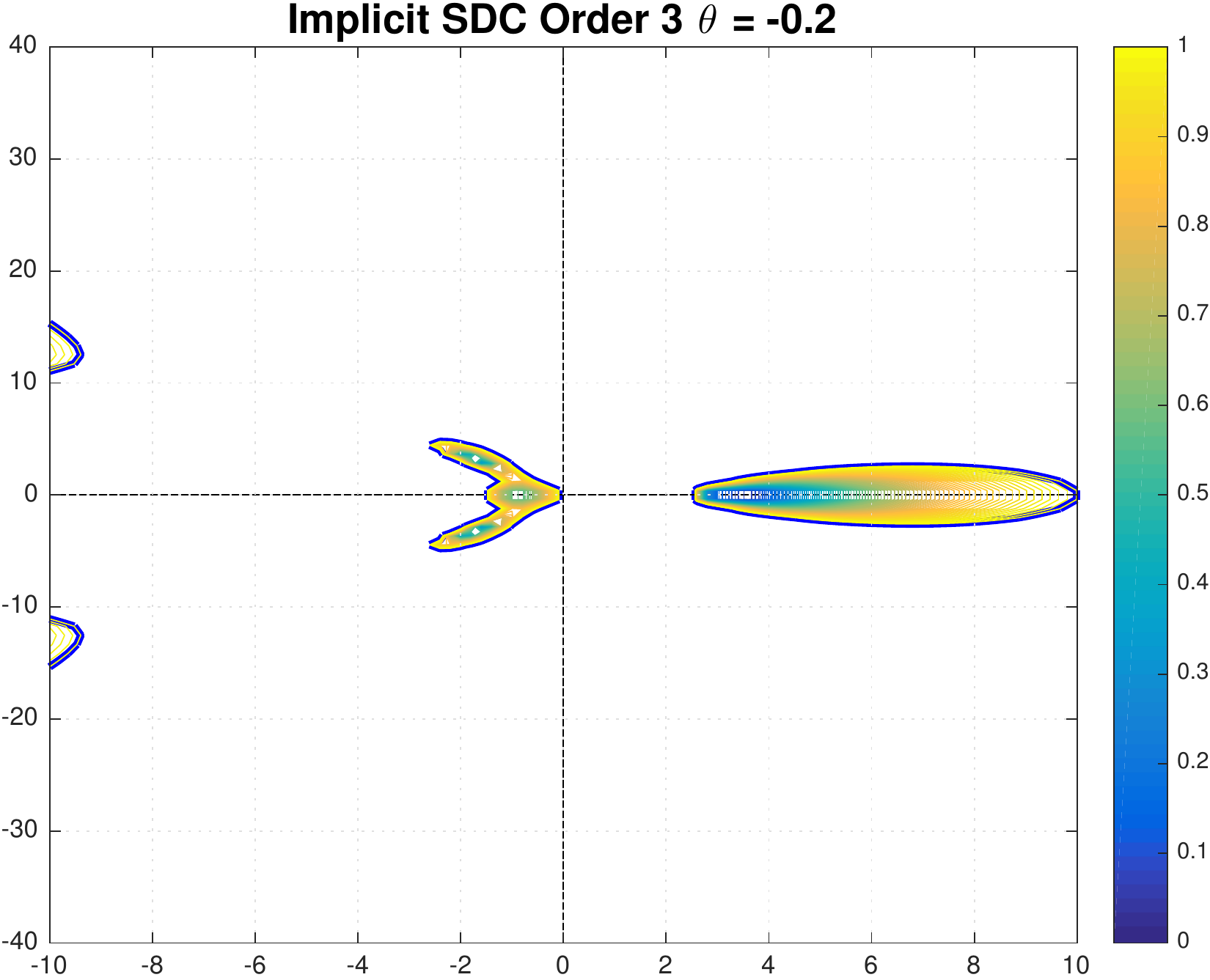}
        \includegraphics[width=0.3\textwidth]{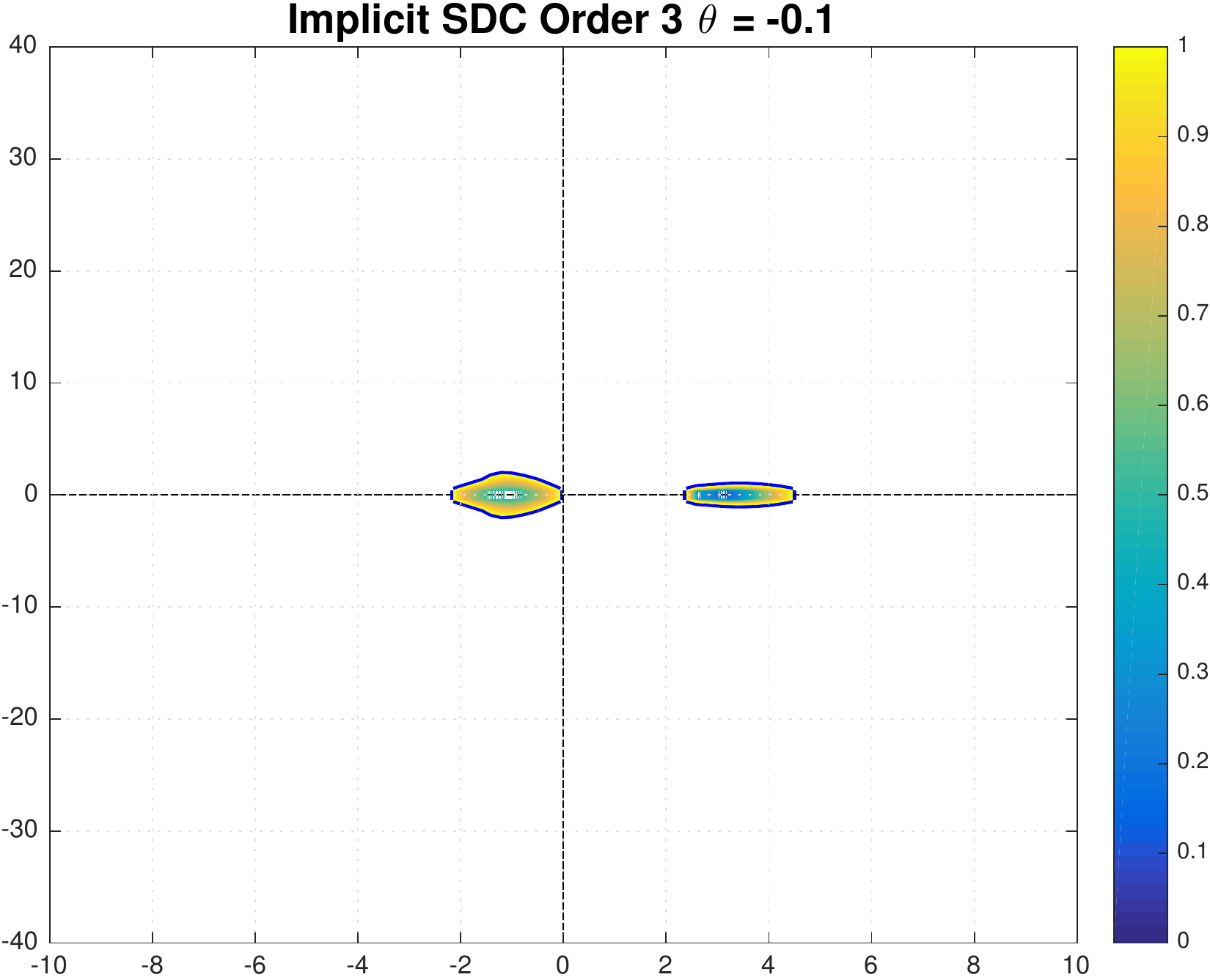} \\
        \includegraphics[width=0.3\textwidth]{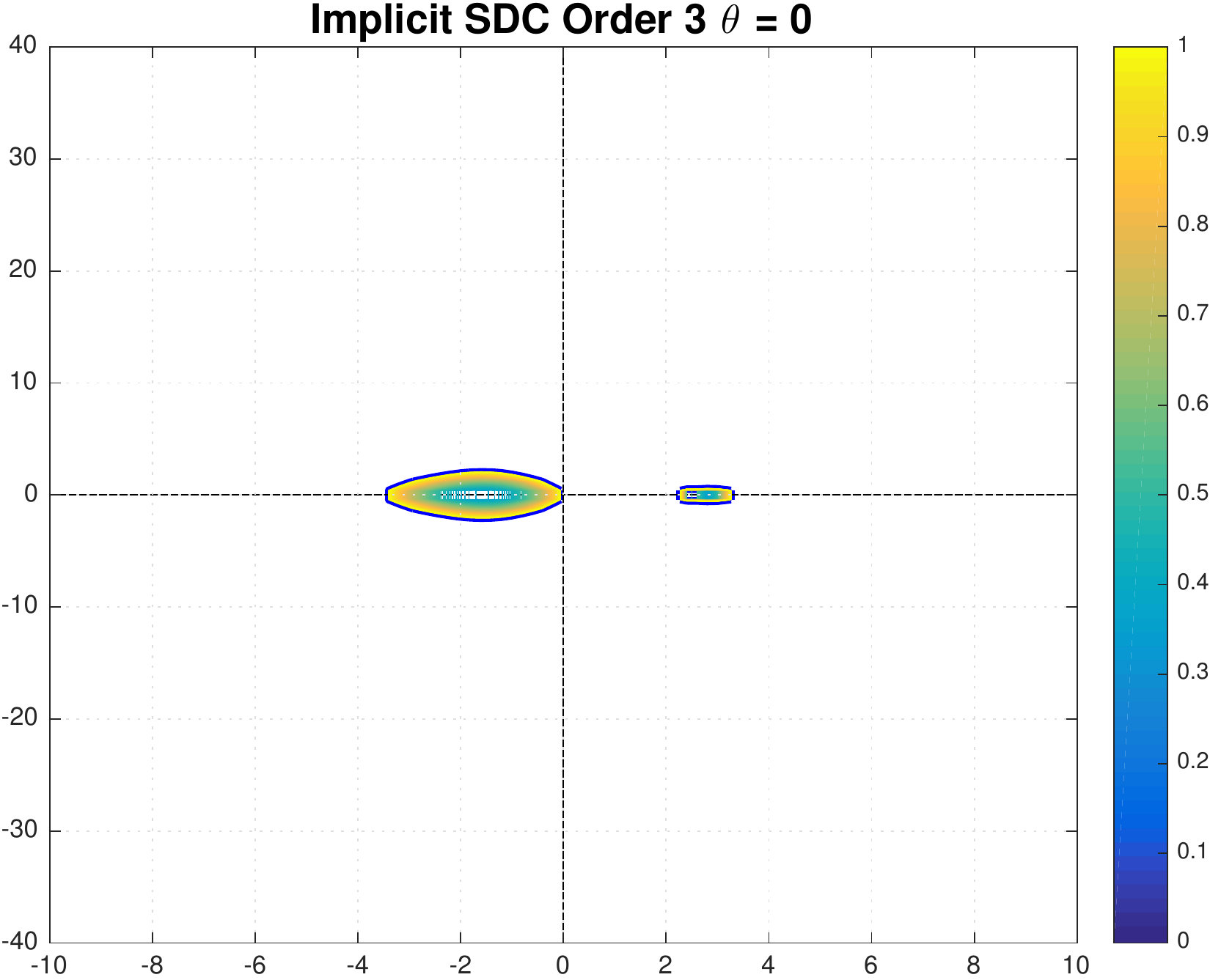}
        \includegraphics[width=0.3\textwidth]{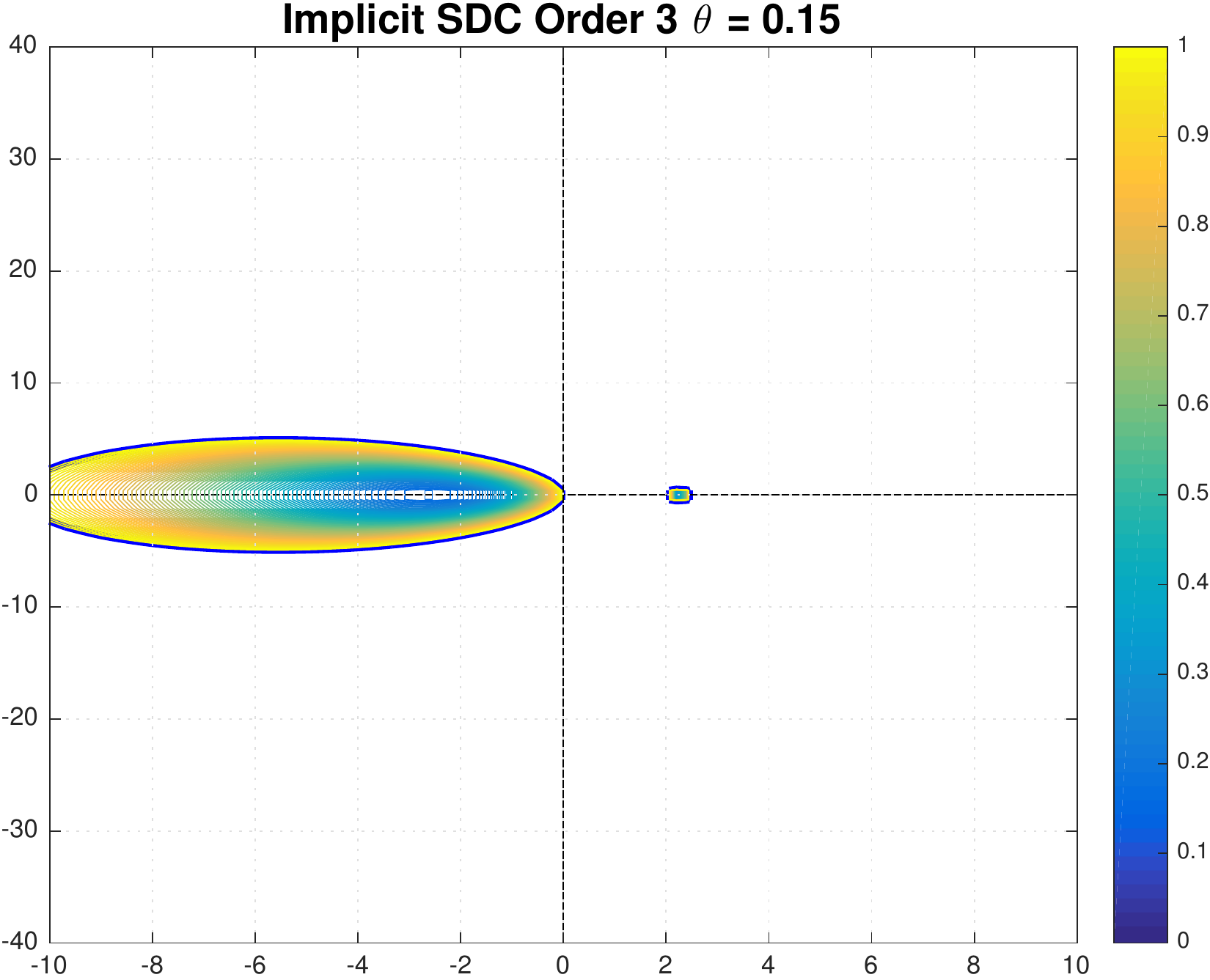}
        \includegraphics[width=0.3\textwidth]{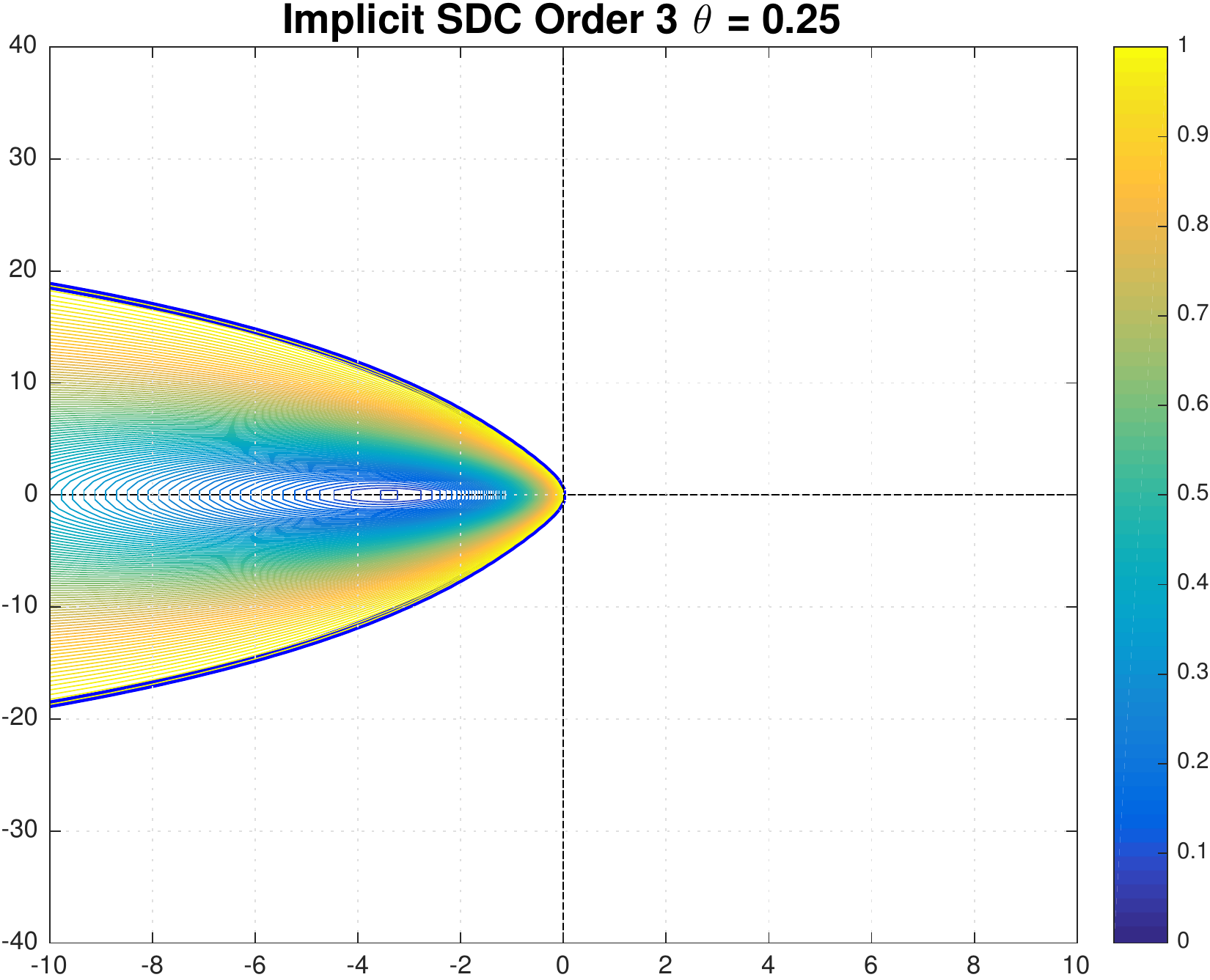} \\
        \includegraphics[width=0.3\textwidth]{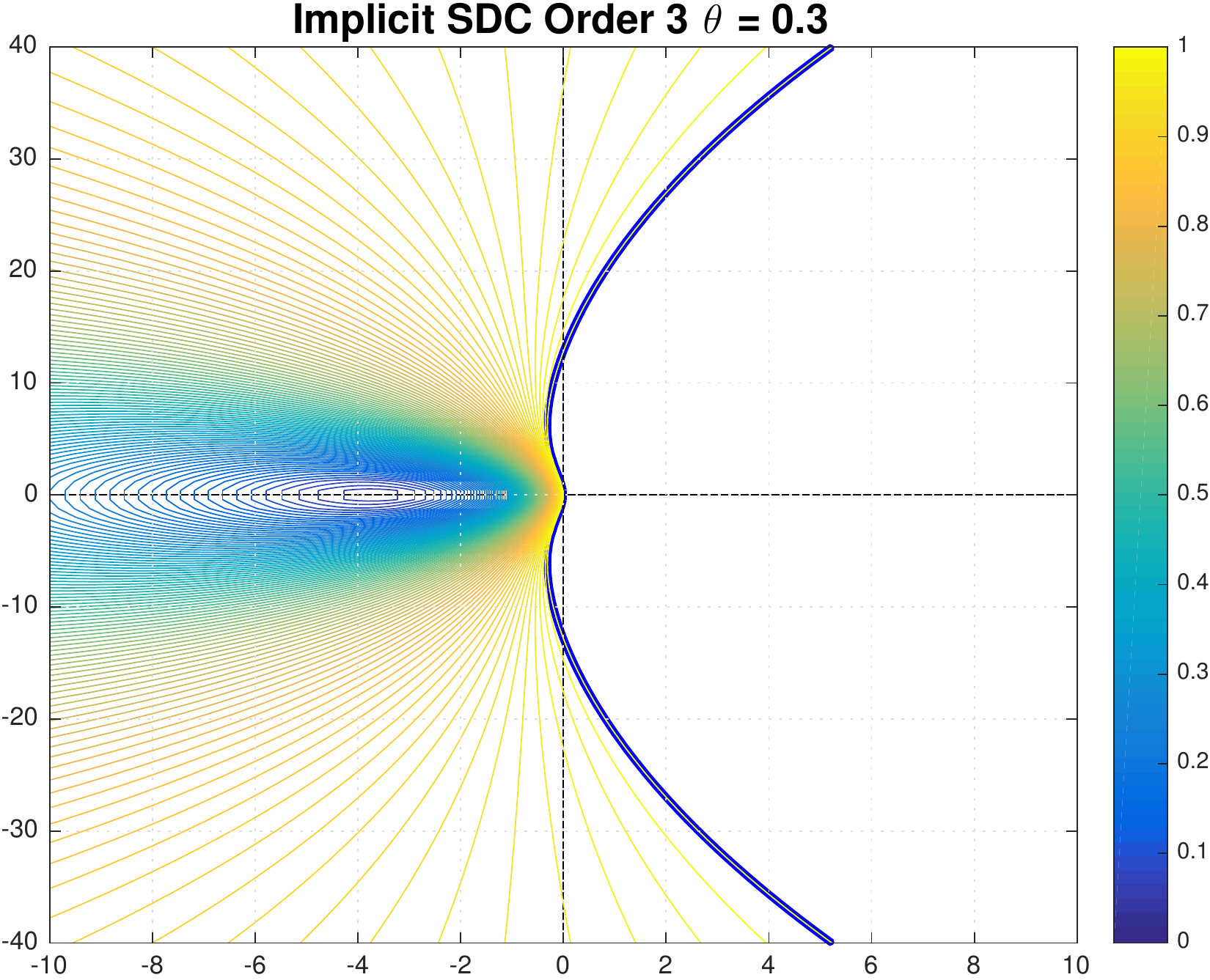}
        \includegraphics[width=0.3\textwidth]{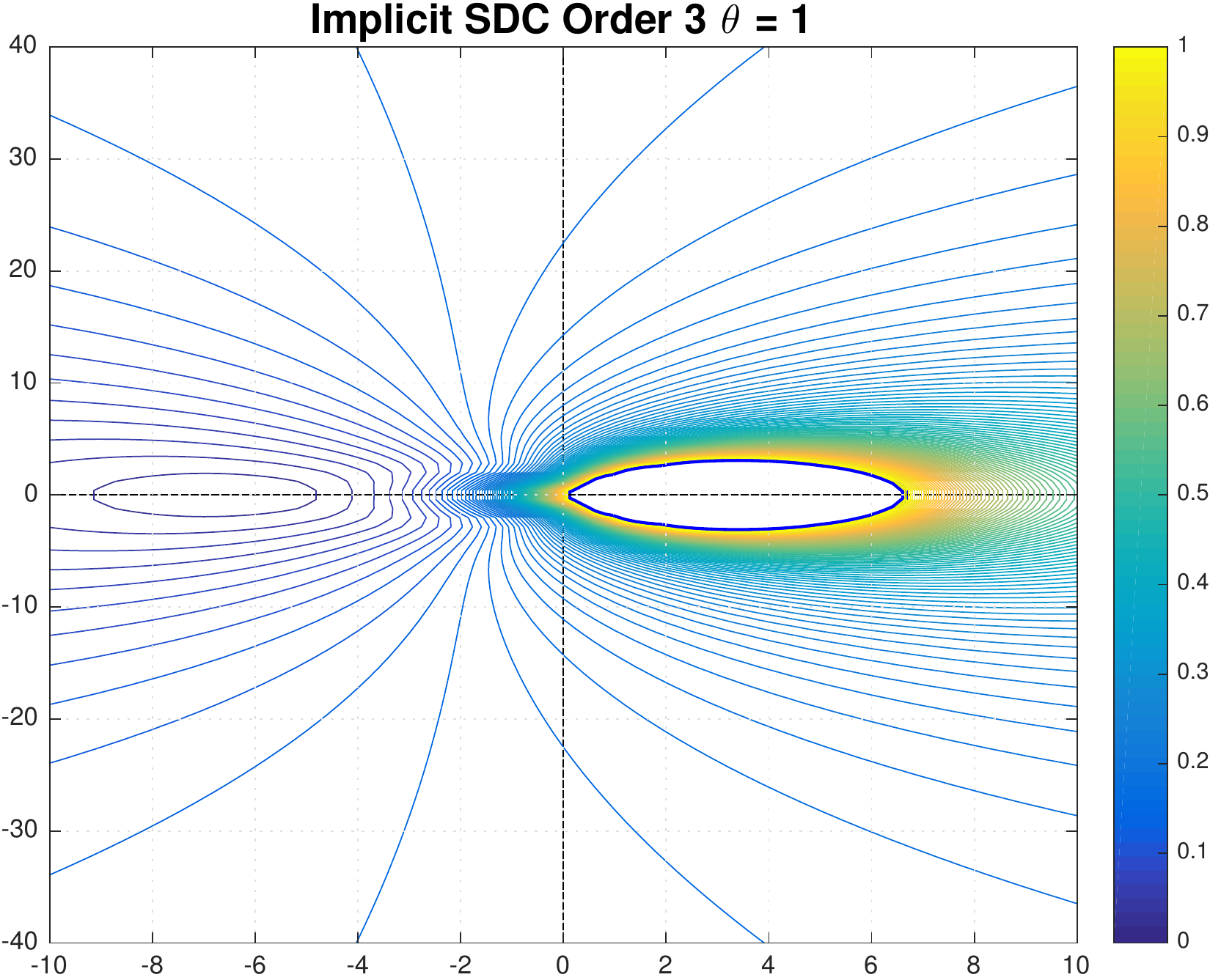}
        \includegraphics[width=0.3\textwidth]{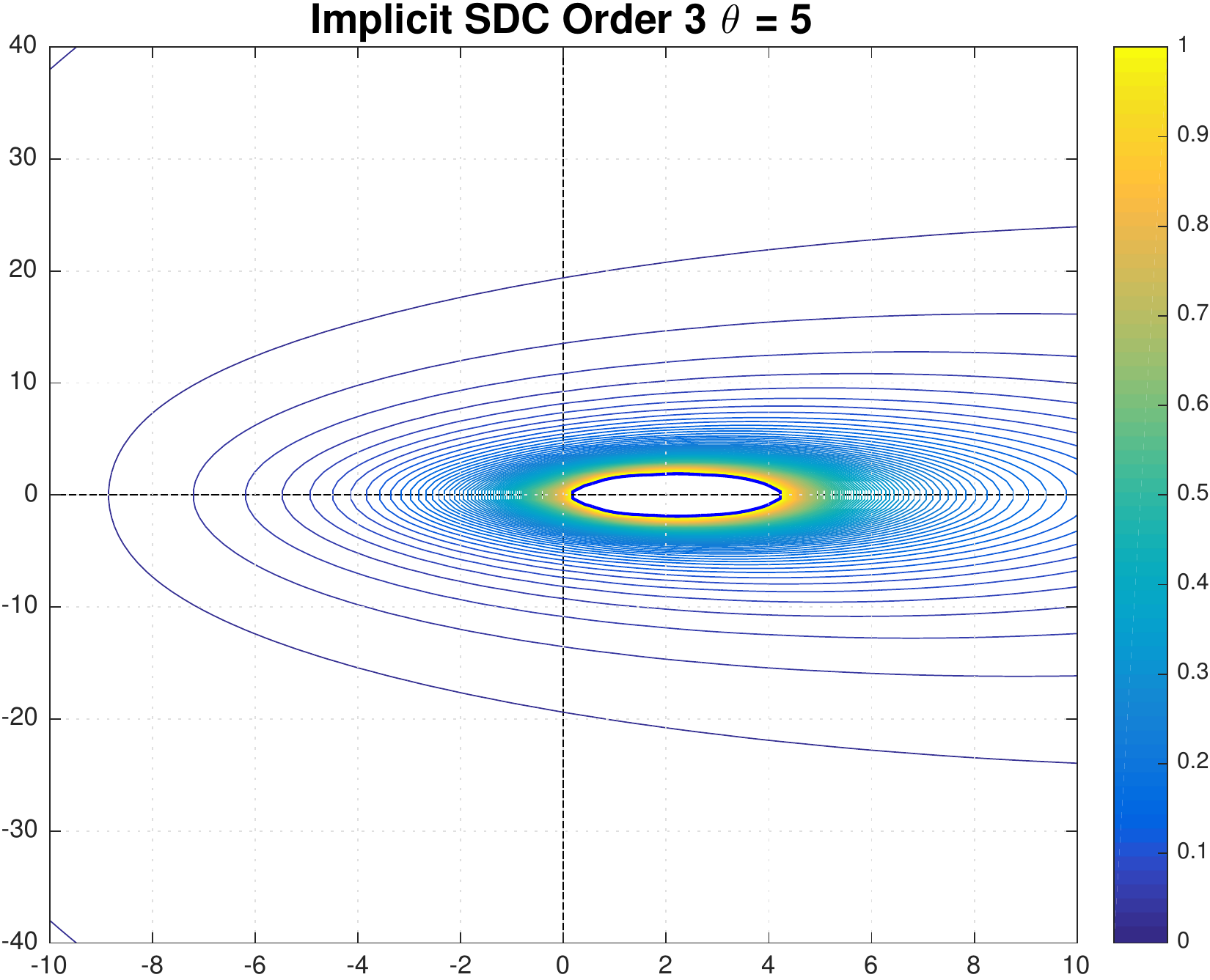} \\
    \end{tabular}
\caption{Stability regions for modified implicit methods.  Here, we perform a parameter study to the modified implicit SDC method
introduced in Eqn.\,\eqref{eqn:implicit-mod-be}; the case with $\theta = 0$ is explicit for each of the correction steps because there is no backward Euler time step to be
solved for, however the stability region is different than the same order Picard method because the initial guess (provisional solution) is computed using 
backward Euler time steps, whereas the Picard method makes use of forward Euler steps for its initial guess.  
Somewhere in the interval $\theta \in [0,1]$ there is a transition between a finite and infinite regions of absolute stability.
Infinite, and large regions are typically desirable when performing implicit solves.
Large values of $\theta$ increase the regions of absolute stability but at the cost of stiff inversions and larger error constants.
\label{fig:stability-plots-imp}}
\end{center}
\end{figure}

\subsection{Modified semi-implicit SDC methods}

Recall that the semi-implicit SDC (SISDC) method begins with a partition of the right hand side into two functions
$f_I$ and $f_E$ via
\begin{equation}
y' = f(y), \quad f(y) = f_I(y) + f_E(y), \quad y(0) = y_0,
\end{equation}
and then they apply a forward Euler/backward Euler (FE/BE) pair to the right hand side
defined in \eqref{eqn:sisdc-classical}:
\begin{equation}
\begin{aligned}
    \eta_n^{[p+1]} &= \eta_{n-1}^{[p+1]} + h_n
    \left[  f_I(\eta_{n}^{[p+1]})-f_I(\eta_{n}^{[p]})\right] \\
    &+ h_n 
    \left[  f_E(\eta_{n-1}^{[p+1]})-f_E(\eta_{n-1}^{[p]})\right] + 
    h \sum_{m=1}^M w_{n,m}f(\eta_m^{[p]}).
\end{aligned}
\end{equation}

With a straightforward extension the results from the present work point out that
high-order accuracy  can be achieved
where there are no forward Euler time steps on the explicit term $f_E(y)$.  That is,
we propose examining the following modified SISDC method:
\begin{equation}
\label{eqn:sisdc-new}
    \eta_n^{[p+1]} = \eta_{n-1}^{[p+1]} + h_n
    \left[  f_I(\eta_{n}^{[p+1]})-f_I(\eta_{n}^{[p]})\right] + h \sum_{m=1}^M w_{n,m}f(\eta_m^{[p]}).
\end{equation}
Note that \emph{the ``base solver" for this method is not even consistent with the underlying ODE.}
This serves as another example of how the findings from this work permit modifications to
classical SDC methods in order to produce
nearby variations.  
As an additional benefit, this reduces the computational coding complexity by asking the user to only
define $f$ and $f_I$ as opposed to $f$, $f_I$ and $f_E$.  This type of modification can be readily found after understanding
the source of high-order accuracy inherent to the SDC framework.

Given that the iterative Picard methods demonstrate smaller errors by dropping the forward Euler terms in the right hand of the iterations from the SDC solver,
one might expect that this method has better accuracy than its SISDC parent.  In the event when $f_I = 0$ (or is small), this would be true because
in that case we would be comparing explicit SDC to Picard iteration, and we have already shown that those errors are smaller for some problems.
However, we will shortly see that
this is not necessarily the case.  Even though this modification
does not affect the overall order of the solver, we will show that for the following test case it does not improve the total overall error of the solver.
With that being said, we believe it is still important to understand the source of the overall order of the SDC solvers, because only then
can new methods be developed from the existing framework.

\subsubsection{Van der Pol's equation}

As a prototypical IMEX example, we include results for Van der Pol's equation:
\begin{equation}
x''(t) = -x(t) + \mu \left( 1 - x(t) \right)^2 x'(t)
\end{equation}
with appropriate initial conditions.  After making the usual transformation of $y_1(t) = x(t)$,
$y_2(t) = \mu x'(t)$, and rescaling time through $t \rightarrow t/\mu$, we have the following
system of differential equations \cite{Minion03,Layton08}:
\begin{equation}
y_1' = y_2, \quad y_2' = \frac{ -y_1 + (1 - y_1^2 ) y_2 }{ \eps },
\quad \eps = \frac{1}{\mu^2}.
\end{equation}
In an IMEX setting, this problem is typically split into $f_E( y ) = \left( y_2, 0 \right)$
and $f_I = \left( 0, \left( -y_1 + (1 - y_1^2 ) y_2 \right)/\eps \right)$
as an effort to account for the stiffness as $\eps \to 0$.
For this problem we only seek to verify the high-order accuracy of 
the classical semi-implicit SDC method defined in\,\eqref{eqn:sisdc-classical}, denoted by SISDC, as well as the modified solver 
defined in \eqref{eqn:sisdc-new}, denoted by ``modified SISDC."  With this aim in mind, we set $\eps = 1$ so that the equations remain non-stiff, and we integrate
to a long final time of $T = 4$.  The initial conditions are the same as those found in an example in \cite{Layton08}, which are
$y_1(0) = 2$ and $y_2(0) = -0.666666654321$.
In Table \ref{tab:vdp-convergence}, we compare a convergence study for the fourth-order
versions of these two methods where we use a total of four equispaced quadrature points, a provisional
solution defined by the split forward/backward Euler method, as well as three corrections
in the solver.  For this problem, we find that the classical SISDC method has slightly smaller errors, despite what
theory might otherwise predict we could observe.  For problems where $f_I$ is negligible or small, the modified
method should outperform the SISDC solver.
Other values for $\epsilon$ produce similar findings, where the usual order reduction can be found as $\epsilon$ approaches zero.
In other cases, the two solvers have similar behavior, and both show high order accuracy for large values of $\epsilon$.  For brevity, these other results are omitted.


\begin{table}
\begin{center}
\caption{Van der Pol oscillator.    Here we present numerical results where we compare 
the implicit classical method defined in \eqref{eqn:sisdc-classical} as well as the modified semi-implicit SDC method
defined in \eqref{eqn:sisdc-new} against each other.  Despite the fact that 
theory can show that the errors could be smaller for the modified method that relies solely on backward Euler time stepping
embedded within Picard iteration, in this case the classical SISDC
method based forward/backward Euler time stepping outperforms the other solver with its smaller error constants.
\label{tab:vdp-convergence}}

\begin{tabular}{|r||c|c||c|c|}
\hline
\bf{Mesh} & {\bf SISDC} & \bf{Order} & {\bf Modified SISDC} & \bf{Order}\\
\hline
\hline
$   4$ & $2.24\times 10^{-02}$ & --- & $6.45\times 10^{-02}$ & ---\\
\hline
$   8$ & $6.06\times 10^{-04}$ & $5.21$ & $2.84\times 10^{-03}$ & $4.51$\\
\hline
$  16$ & $4.11\times 10^{-05}$ & $3.88$ & $1.91\times 10^{-04}$ & $3.89$\\
\hline
$  32$ & $3.44\times 10^{-06}$ & $3.58$ & $1.46\times 10^{-05}$ & $3.71$\\
\hline
$  64$ & $2.56\times 10^{-07}$ & $3.75$ & $1.01\times 10^{-06}$ & $3.85$\\
\hline
$ 128$ & $1.78\times 10^{-08}$ & $3.85$ & $6.68\times 10^{-08}$ & $3.92$\\
\hline
$ 256$ & $1.17\times 10^{-09}$ & $3.93$ & $4.29\times 10^{-09}$ & $3.96$\\
\hline
$ 512$ & $7.26\times 10^{-11}$ & $4.01$ & $2.69\times 10^{-10}$ & $3.99$\\
\hline
\end{tabular}
\end{center}
\end{table}


\section{Conclusions}
\label{section:conclusions}

In this work we present rigorous error bounds for both explicit and implicit spectral deferred correction methods.
Unlike most presentations that introduce SDC methods as a method that iteratively corrects provisional solutions
by solving an error equation, our work hinges on the fact that the basic solver can be recast as a variation on Picard iteration.
This observation allows new SDC methods to be developed through modifications of the (forward or backward) Euler part of the
iterative procedure.  In addition, we present some analysis for SDC methods constructed with higher order base solvers.  
In the numerical results section we present some sample variations that serve to indicate that the choice of the base solver
need not be consistent with the underlying ODE in order to obtain a method that converges.
That is, our findings indicate that it is not important to use the same low-order solver for each correction step because the
desired high-order accuracy can be found in the integral of the residual.
However, the choice of the base solver certainly has an impact on the overall scheme.  For example, up to the degree of precision of
the underlying quadrature rule, the choice of the forward or backward
Euler method or even an inconsistent base solver leads to a single pickup on the order of accuracy of the solver with each correction step, whereas 
the choice of a trapezoidal rule for a base solver yields two orders of pickup with each correction step.
For stiff problems, an implicit method
constructed with the backward Euler method (or some variation of it) is certainly preferable due to the larger regions of absolute stability.
Future work involves further analysis of embedded high-order base solvers as well as exploring further modifications of the solver to modify regions of absolute stability
of existing solvers for explicit, implicit, and semi-implicit SDC methods. \\

\noindent {\bf Acknowledgements.}  We would like to thank the anonymous referees for their thoughtful comments, suggestions that improved the quality of this manuscript,
and recommendations for future research.
The work of D.C. Seal was supported by the Naval Academy Research Council.

\newpage
\appendix


\bibliography{convergence-proofs-sdc_v7}

\end{document}